\documentclass[12pt]{amsart}
\newcommand{\deleted}[1]{}
\newcommand{\delete}[1]{}
\newcommand{\mynotes}[1]{}
\newcommand\notes[1]{}
\usepackage{txfonts}
\usepackage{amssymb}
\usepackage{eucal}
\usepackage{graphicx}
\usepackage{amsmath}
\usepackage{amscd}
\usepackage[all]{xy}

\usepackage{amsfonts,latexsym}
\usepackage{xspace}
\usepackage{epsfig}
\usepackage{float}
\usepackage{color}
\usepackage{fancybox}
\usepackage{colordvi}
\usepackage{multicol}
\usepackage{tikz}
\usepackage{wasysym}
\usepackage{xspace}
\usepackage[active]{srcltx} 
\usepackage{mathtools} 

\newcommand\changed[1]{#1}
\changed{}

\newtheorem{theorem}{Theorem}[section]
\newtheorem{lemma}[theorem]{Lemma}

\newtheorem{prop}[theorem]{Proposition}
\theoremstyle{definition}
\newtheorem{defn}[theorem]{Definition}
\newtheorem{remark}[theorem]{Remark}

\deleted{\newtheorem{theorem}{Theorem}[section]
\newtheorem{prop-def}{Proposition-Definition}[section]
\newtheorem{coro-def}{Corollary-Definition}[section]

}

\newcommand{\nc}{\newcommand}
\nc{\tred}[1]{\textcolor{red}{#1}} \nc{\tblue}[1]{\textcolor{blue}{#1}} \nc{\tgreen}[1]{\textcolor{green}{#1}} \nc{\tpurple}[1]{\textcolor{purple}{#1}} \nc{\btred}[1]{\textcolor{red}{\bf #1}} \nc{\btblue}[1]{\textcolor{blue}{\bf #1}} \nc{\btgreen}[1]{\textcolor{green}{\bf #1}} \nc{\btpurple}[1]{\textcolor{purple}{\bf #1}}

\renewcommand{\Bbb}{\mathbb}

\newcommand{\efootnote}[1]{}

\newcommand\wyscco[1]{}
\renewcommand{\textbf}[1]{}

\nc{\mlabel}[1]{\label{#1}}  
\nc{\mcite}[1]{\cite{#1}}  
\nc{\mref}[1]{\ref{#1}}  
\nc{\mbibitem}[1]{\bibitem{#1}} 

\delete{
\nc{\mlabel}[1]{\label{#1}{\hfill \hspace{1cm}{\bf{{\ }\hfill(#1)}}}}
\nc{\mcite}[1]{\cite{#1}{{\bf{{\ }(#1)}}}}  
\nc{\mref}[1]{\ref{#1}{{\bf{{\ }(#1)}}}}  
\nc{\mbibitem}[1]{\bibitem[\bf #1]{#1}} 
}

\renewcommand\geq{\geqslant}
\renewcommand\leq{\leqslant}

\renewcommand\bar[1]{\overline{#1}}
\renewcommand\tilde[1]{\widetilde{#1}}

\nc\kdot{\bfk}
\nc\simple{simple\xspace}

\nc{\rbw}{\mathfrak{R}} \nc{\brp}{\mathrm{brp}} \nc{\lead}{\mathrm{Lead}} \nc{\Id}{\mathrm{Id}} \nc{\Irr}{\mathrm{Irr}} \nc{\vx}{\sigma} \nc{\vy}{\tau} \nc{\dvx}{\sigma^{(1)}} \nc{\dvy}{\tau^{(1)}} \nc{\done}{\vep} \nc{\citep}[1]{\cite{#1}} \nc{\wt}{\mathrm{wt}} \nc{\bre}[1]{|#1|} \nc{\mapmonoid}{\frakM} \nc{\disjoint}{\frakM'}
\nc{\ncpoly}[1]{\langle #1\rangle}  
\nc{\mapm}[1]{\frakM(#1)}
\nc{\diff}[1]{{}^\NC\{ #1 \}} \nc{\disj}[1]{\{{#1}\}'} \nc{\mdisj}[1]{\frakM'(#1)} \nc{\brho}{\bar{\rho}} \nc{\om}{\bar{\frakm}} \nc{\frakn}{\mathfrak n} \nc{\ddeg}[1]{^{(#1)}} \nc{\opset}{X} \nc{\genset}{{Z}} \nc{\NC}{\mathrm{{NC}}} \nc{\leaf}{\mathrm{leaf}} \nc{\twig}{\mathrm{twig}} \nc{\fe}{\mathrm{fl}} \nc{\munderline}[1]{#1} \nc{\bo}{o} \nc{\dep}{\mathrm{dep}} \nc{\ofe}{\mathrm{ofl}} \nc{\dfe}{\mathrm{dfe}} \nc{\fex}{\mathrm{fex}} \nc{\dl}{\mathrm{dlex}} \nc{\db}{\mathrm{db}} \nc{\lex}{\mathrm{lex}} \nc{\clex}{\mathrm{clex}} \nc{\dgp}{\mathrm{dgp}} \nc{\dgx}{\mathrm{dgx}} \nc{\br}{\mathrm{br}} \nc{\obd}{\mathrm{odb}} \nc{\ob}{\mathrm{ob}}
\nc{\loc}{location\xspace}
\nc{\occ}{occurrence\xspace}
\nc{\occs}{occurrences\xspace}
\nc{\pla}{placement\xspace}
\nc{\plas}{placements\xspace}

\nc{\bin}[2]{ (_{\stackrel{\scs{#1}}{\scs{#2}}})}  
\nc{\binc}[2]{ \left (\!\! \begin{array}{c} \scs{#1}\\
    \scs{#2} \end{array}\!\! \right )}  
\nc{\bincc}[2]{  \left ( {\scs{#1} \atop
    \vspace{-1cm}\scs{#2}} \right )}  
\nc{\bs}{\bar{S}} \nc{\cosum}{\sqsubset} \nc{\la}{\longrightarrow} \nc{\rar}{\rightarrow} \nc{\dar}{\downarrow} \nc{\dprod}{**} \nc{\dap}[1]{\downarrow \rlap{$\scriptstyle{#1}$}} \nc{\md}{\mathrm{dth}} \nc{\uap}[1]{\uparrow \rlap{$\scriptstyle{#1}$}} \nc{\defeq}{\stackrel{\rm def}{=}} \nc{\disp}[1]{\displaystyle{#1}} \nc{\dotcup}{\ \displaystyle{\bigcup^\bullet}\ } \nc{\gzeta}{\bar{\zeta}} \nc{\hcm}{\ \hat{,}\ } \nc{\hts}{\hat{\otimes}} \nc{\barot}{{\otimes}} \nc{\free}[1]{\widetilde{#1}} \nc{\uni}[1]{\tilde{#1}} \nc{\hcirc}{\hat{\circ}} \nc{\leng}{\ell} \nc{\lleft}{[} \nc{\lright}{]} \nc{\lc}{\lfloor} \nc{\rc}{\rfloor}
\nc{\lb}{[} 
\nc{\rb}{]} 
\nc{\curlyl}{\left \{ \begin{array}{c} {} \\ {} \end{array}
    \right.  \!\!\!\!\!\!\!}
\nc{\curlyr}{ \!\!\!\!\!\!\!
    \left. \begin{array}{c} {} \\ {} \end{array}
    \right \} }
\nc{\longmid}{\left | \begin{array}{c} {} \\ {} \end{array}
    \right. \!\!\!\!\!\!\!}
\nc{\onetree}{\bullet} \nc{\ora}[1]{\stackrel{#1}{\rar}}
\nc{\ola}[1]{\stackrel{#1}{\la}}
\nc{\ot}{\otimes} \nc{\mot}{{{\boxtimes\,}}} \nc{\otm}{\overline{\boxtimes}} \nc{\sprod}{\bullet} \nc{\scs}[1]{\scriptstyle{#1}} \nc{\mrm}[1]{{\rm #1}} \nc{\msum}{\sum\limits}
\nc{\margin}[1]{\marginpar{\rm #1}}   
\nc{\dirlim}{\displaystyle{\lim_{\longrightarrow}}\,} \nc{\invlim}{\displaystyle{\lim_{\longleftarrow}}\,} \nc{\mvp}{\vspace{0.3cm}} \nc{\tk}{^{(k)}} \nc{\tp}{^\prime} \nc{\ttp}{^{\prime\prime}} \nc{\svp}{\vspace{2cm}} \nc{\vp}{\vspace{8cm}} \nc{\proofbegin}{\noindent{\bf Proof: }}
\nc{\proofend}{$\blacksquare$ \vspace{0.3cm}}
\nc{\modg}[1]{\!<\!\!{#1}\!\!>}
\nc{\intg}[1]{F_C(#1)} \nc{\lmodg}{\!<\!\!} \nc{\rmodg}{\!\!>\!} \nc{\cpi}{\widehat{\Pi}}
\nc{\sha}{{\mbox{\cyr X}}}  
\nc{\shap}{{\mbox{\cyrs X}}} 
\nc{\shpr}{\diamond}    
\nc{\shp}{\ast} \nc{\shplus}{\shpr^+}
\nc{\shprc}{\shpr_c}    
\nc{\msh}{\ast} \nc{\zprod}{m_0} \nc{\oprod}{m_1} \nc{\vep}{\varepsilon} \nc{\labs}{\mid\!} \nc{\rabs}{\!\mid}
\nc{\astarrow}{\overset{\raisebox{-2pt}{{\scriptsize $\ast$}}}{\rightarrow}}
\nc{\astlarrow}{\overset{\raisebox{-2pt}{{\scriptsize $\ast$}}}{\longrightarrow}}
\nc{\lastarrow}{\overset{\raisebox{-2pt}{{\scriptsize $\ast$}}}{\leftarrow}}
\nc{\mastarrow}[1]{\overset{\raisebox{-2pt}{{\scriptsize $#1$}}}{\rightarrow}}
\nc{\quvarrow}[3]{#1 \overset{q,u,v}{\longrightarrow}_{#3} #2}
\nc{\quvkto}[1]{f_{#1} \overset{q_{#1}, u_{#1}, v_{#1}}{\longrightarrow}_\phi g_{#1}}
\nc{\tvarrow}[3]{#1 \overset{(t,v)}{\longrightarrow}_{#3} #2}
\nc{\Supp}{{\rm Supp}}
\nc{\mpu}{u^{\ast}}
\nc{\mpv}{v^{\ast}}
\nc{\mpw}{w^{\ast}}
\nc{\mpx}{x^{\ast}}
\nc{\dps}{\dotplus}

\nc{\dth}{d} \nc{\mmbox}[1]{\mbox{\ #1\ }} \nc{\fp}{\mrm{FP}} \nc{\rchar}{\mrm{char}} \nc{\Fil}{\mrm{Fil}} \nc{\Mor}{Mor\xspace} \nc{\gmzvs}{gMZV\xspace} \nc{\gmzv}{gMZV\xspace} \nc{\mzv}{MZV\xspace} \nc{\mzvs}{MZVs\xspace} \nc{\Hom}{\mrm{Hom}} \nc{\id}{\mrm{id}} \nc{\im}{\mrm{im}} \nc{\incl}{\mrm{incl}} \nc{\map}{\mrm{Map}} \nc{\mchar}{\rm char} \nc{\nz}{\rm NZ} \nc{\supp}{\mathrm Supp}
\nc{\mo}{\mathbf o}
\nc{\pl}{\mathfrak{p}}

\nc{\Alg}{\mathbf{Alg}} \nc{\Bax}{\mathbf{Bax}} \nc{\bff}{\mathbf f} \nc{\bfk}{{\bf k}} \nc{\bfone}{{\bf 1}} \nc{\bfx}{\mathbf x} \nc{\bfy}{\mathbf y}
\nc{\base}[1]{\bfone^{\otimes ({#1}+1)}} 
\nc{\Cat}{\mathbf{Cat}} \delete{}
\nc{\detail}{\marginpar{\bf More detail}
    \noindent{\bf Need more detail!}
    \svp}
\nc{\Int}{\mathbf{Int}} \nc{\Mon}{\mathbf{Mon}}
\nc{\rbtm}{{shuffle }} \nc{\rbto}{{Rota-Baxter }} \nc{\remarks}{\noindent{\bf Remarks: }} \nc{\Rings}{\mathbf{Rings}} \nc{\Sets}{\mathbf{Sets}}
\nc{\vwpt}{{Let $V$ be a free $\bfk$-module with a $\bfk$-basis $W$ and let $\Pi$ be a \simple term-rewriting system on $V$ with respect to $W$.}\xspace}

\nc{\BA}{{\Bbb A}} \nc{\CC}{{\Bbb C}} \nc{\DD}{{\Bbb D}} \nc{\EE}{{\Bbb E}} \nc{\FF}{{\Bbb F}} \nc{\GG}{{\Bbb G}} \nc{\HH}{{\Bbb H}} \nc{\LL}{{\Bbb L}} \nc{\NN}{{\Bbb N}} \nc{\KK}{{\Bbb K}} \nc{\QQ}{{\Bbb Q}} \nc{\RR}{{\Bbb R}} \nc{\TT}{{\Bbb T}} \nc{\VV}{{\Bbb V}} \nc{\ZZ}{{\Bbb Z}}

\nc{\cala}{{\mathcal A}} \nc{\calc}{{\mathcal C}} \nc{\cald}{{\mathcal D}} \nc{\cale}{{\mathcal E}} \nc{\calf}{{\mathcal F}} \nc{\calg}{{\mathcal G}} \nc{\calh}{{\mathcal H}} \nc{\cali}{{\mathcal I}} \nc{\call}{{\mathcal L}} \nc{\calm}{{\mathcal M}} \nc{\caln}{{\mathcal N}} \nc{\calo}{{\mathcal O}} \nc{\calp}{{\mathcal P}} \nc{\calr}{{\mathcal R}} \nc{\cals}{{\mathcal S}} \nc{\calt}{{\mathcal T}} \nc{\calw}{{\mathcal W}}
\nc{\calv}{{\mathcal V}}
\nc{\calk}{{\mathcal K}} \nc{\calx}{{\mathcal X}} \nc{\CA}{\mathcal{A}}

\nc{\fraka}{{\mathfrak a}} \nc{\frakA}{{\mathfrak A}} \nc{\frakb}{{\mathfrak b}} \nc{\frakB}{{\mathfrak B}} \nc{\frakC}{{\mathfrak C}}
\nc{\frakD}{{\mathfrak D}} \nc{\frakH}{{\mathfrak H}} \nc{\frakM}{{\mathfrak M}} \nc{\bfrakM}{\overline{\frakM}} \nc{\frakm}{{\mathfrak m}} \nc{\frkP}{{\mathfrak P}}
\nc{\frakN}{{\mathfrak N}} \nc{\frakp}{{\mathfrak p}} \nc{\fraku}{{\mathfrak u}} \nc{\frakv}{{\mathfrak v}}
\nc{\frakQ}{{\mathfrak Q}}\nc{\frakR}{{\mathfrak R}} \nc{\frakS}{{\mathfrak S}}
\nc{\frakx}{{\mathfrak x}} \nc{\ox}{\bar{\frakx}} \nc{\frakX}{{\mathfrak X}} \nc{\fraky}{{\mathfrak y}}
\nc\dop{\delta}
\nc{\Reduce}{{\rm Red}}

\font\cyr=wncyr10 \font\cyrs=wncyr7
\nc{\redt}[1]{\textcolor{red}{#1}}

\nc{\jing}[1]{\textcolor{red}{Jing:#1}} 
\nc{\lio}[1]{}
\nc{\sz}[1]{\textcolor{green}{sz:#1}}
\nc{\szo}[1]{}
\nc{\xing}[1]{\textcolor{purple}{Xing:#1}}
\nc{\ws}[1]{\textcolor{blue}{{#1}}} 
\nc{\wsc}[1]{\textcolor{blue}{ws:#1}} 
\nc{\wsco}[1]{}
\nc{\wsn}[1]{\textcolor{magenta}{#1}} 

\nc{\medmid}{{\,~{\tiny \longmid}~\,}}
 \nc{\lbar}[1]{\overline{#1}}
\nc{\anf}{\bf $\phi$-NF} \nc{\lto}{\longrightarrow}
\nc{\gs}{Gr\"{o}bner-Shirshov\xspace}
\nc{\Gs}{Gr\"{o}bner-Shirshov\xspace}
\nc{\pgs}{potentially Gr\"{o}bner-Shirshov\xspace}
\nc{\Pgs}{Potentially Gr\"{o}bner-Shirshov\xspace}
\nc{\egs}{essentially Gr\"obner-Shirshov\xspace}
\nc{\cs}{convergent\xspace}
\nc{\pcs}{potentially convergent\xspace}
\nc{\ecs}{essentially convergent\xspace}

\nc{\brw}{\frakM(Z)} \nc{\irr}{{\rm Irr}} \nc{\pis}{\Pi_S}
\nc{\term}{term\xspace} \nc{\Term}{Term\xspace}
\nc{\re}[1]{R(#1)} \nc{\sumre}[2]{R^{#1}_{#2}}
\nc{\stars}[2]{#1|_{#2}}
\nc{\nbfk}{\bfk^{\times}} \nc{\revise}[1]{\textcolor{blue}{#1}}
\nc{\ord}{{\rm ord}} \nc{\tpi}{\to_{\Pi}}
\nc{\fix}[1]{\tilde{#1}} \nc{\topi}{\to_{\Pi_S}} \nc{\astarrowpi}{\astarrow_{\Pi_S}}
\nc{\pre}{\calp} \nc{\ocong}[1]{\langle #1\rangle}

\begin{document}
\title[Free operated monoids and Rewriting systems]{Free operated monoids and Rewriting systems}
\author{Jin Zhang} \address{School of Mathematics and Statistics,
Lanzhou University, Lanzhou, 730000, P.R. China}
\email{zj$_{-}$10@lzu.edu.cn}

\author{Xing Gao $^{*}$}\thanks{*Corresponding author} \address{School of Mathematics and Statistics,
Key Laboratory of Applied Mathematics and Complex Systems,
Lanzhou University, Lanzhou, 730000, P.R. China}
\email{gaoxing@lzu.edu.cn}

\hyphenpenalty=8000
\date{\today}

\begin{abstract}
The construction of bases for quotients is an important problem. In this paper, applying the method of
rewriting systems, we give a unified approach to construct
sections---an alternative name for bases in semigroup theory---for quotients of free operated monoids.
As applications, we capture sections of free $\ast$-monoids and free groups, respectively.
\end{abstract}

\subjclass[2010]{
16S15, 
08A70, 
06F05,  	
20M05
}

\keywords{
Operated monoids, term-rewriting systems, free $\ast$-monoids, free groups.
}

\maketitle
\vspace{-1.2cm}

\tableofcontents

\vspace{-1.3cm}

\allowdisplaybreaks


\section{Introduction}
In 1960,  A.\,G. Kurosh~\mcite{Ku} first introduced the concept of algebras with one or more linear operators.
An example of such algebras is the differential algebra led by the algebraic abstraction of differential operator
in analysis. Differential algebra is from a purely algebraic viewpoint to study
differentiation and nonlinear differential equations without using an underlying topology, and has been largely successful in many crucial
areas, such as uncoupling of nonlinear systems, classification of singular components, and detection of hidden equations~\mcite{Kol, Rit}.
Another important example of such algebras is the Rota-Baxter algebra (first called Baxter algebra)
which is the algebraic abstraction of integral operator in analysis~\mcite{Gub}.
Rota-Baxter algebra, originated from probability study~\mcite{Ba}, is related beautifully to the classical
Yang-Baxter equation, as well as to operads, to combinatorics and, through the Hopf
algebra framework of Connes and Kreimer, to the renormalization of quantum field theory~\mcite{Agu, Bai, BBGN, CK1, FGK, GZ}.
Other examples are also important, such as averaging algebra, Reynolds algebras, Nijenhuis algebras and Leroux's TD algebras~\mcite{CGM, Mill,Ler}.
Each of the above examples is an algebra with one linear operator, which is named operated
algebras by Guo~\mcite{Gop}.

\begin{defn}
An {\bf operated monoid} (resp. operated \bfk-algebra) is a monoid (resp. \bfk-algebra) $U$ together with a map (resp. \bfk-linear map) $P_U: U\to U$,
where $\bfk$ is a commutative unitary ring.
\label{de:mapset}
\end{defn}

In that paper~\mcite{Gop}, Guo constructed the free operated \bfk-algebra on a set.
Since the free operated \bfk-algebra as modules is precisely the free \bfk-module with basis the free operated monoid,
the crucial step of Guo's method is to construct the free operated monoid
on a set---the main object considered in this paper,
See also~\mcite{BCQ,GG}.

Abstract rewriting system is a branch of theoretical computer science, combining elements of logic, universal algebra, automated theorem proving and functional programming~\mcite{BN, Oh}.
The theory of convergent rewriting systems is successfully applied to find bases of free differential type~\mcite{GSZ} and free Rota-Baxter type algebras~\mcite{GGSZ},
which reveals the power of rewriting systems in the study of operators.

Let us point out that groups can be viewed as operated monoids if one considers the inverse operator as a map from the group to
itself. In the same way, many other important classes of monoids such as inverse monoids, I-monoids and $\ast$-monoids can also be fitted into the
framework of operated monoids. All of these examples can be obtained from free operated monoids by taking quotients modulo suitable operated
congruences. It is interesting to find bases for quotients. The bases of quotients in
semigroup theory are also called sections. In the present paper we obtained a method, in terms of convergent rewriting systems, to give
sections for quotients of free operated monoids. Our method is parallel to the famous Composition-Diamond lemma in Gr\"{o}bner-Shirshov theory~\mcite{BCQ}.
As applications, we capture sections of free $\ast$-monoids and free groups, respectively.

The organisation of this paper is as follows. In Section~\mref{ss:rse}, after reviewing the construction of
free operated monoids, we characterize the operated congruence generated by a binary relation on a free operated monoid (Proposition~\mref{prop}).
Then we associate to each binary relation on a free operated monoid a rewriting system (Definition~\mref{defn:rsts}).
We also establish a relationship between convergent rewriting systems on free operated monoids and sections of quotients of free operated monoids. (Theorem~\mref{thm:sect}). In Section~\mref{ss:sai}, as applications of our main result, we acquire respectively sections of free $\ast$-monoids (Theorem~\ref{thm:section}) and free groups (Theorem~\ref{thm:sectionI}).

\section{Rewrting systems and sections}
\mlabel{ss:rse}
In this section, based on rewriting systems on  free operated monoids $\frakM(X)$, we give an approach to construct
sections for quotients of $\frakM(X)$.

\subsection{Operated monoids and operated congruences}

The construction of free operated monoids was given in~\mcite{Gop, GSZ}. See also~\mcite{BCQ}.
We reproduce that construction here to review the notations. For any set $Y$, denote by $M(Y)$
the free monoid on $Y$.

Let $X$ be a set.  We proceed via finite stages $\frakM_n(X)$ defined recursively to
construct the free operated monoid $\frakM(X)$ on $X$.
The initial stage is $\frakM_0(X) := M(X)$ and $\frakM_1(X) := M(X \cup \lc \frakM_0(X)\rc)$,
where $\lc \frakM_0(X)\rc:= \{ \lc u\rc \mid u\in \frakM_0(X)\}$ is a disjoint copy of $\frakM_0(X)$.
The inclusion $X\hookrightarrow X \cup \lc \frakM_0\rc $ induces a monomorphism
$$i_{0}:\mapmonoid_0(X) = M(X) \hookrightarrow \mapmonoid_1(X) = M(X \cup \lc \frakM_0\rc  )$$
of monoids through which we identify $\mapmonoid_0(X) $ with its image in $\mapmonoid_1(X)$.

For~$n\geq 1$, assume inductively that
$\frakM_{ n}(X)$ has been defined and the embedding
$$i_{n-1,n}\colon  \frakM_{ n-1}(X) \hookrightarrow \frakM_{ n}(X)$$
has been obtained. Then we define
\begin{equation*}
 \label{eq:frakn}
 \frakM_{ n+1}(X) := M \big( X\cup\lc\frakM_{n}(X) \rc\big).
\end{equation*}
Since~$\frakM_{ n}(X) = M \big(X\cup \lc\frakM_{ n-1}(X) \rc\big)$ is a free monoid,
the injection
$$  \lc\frakM_{ n-1}(X) \rc \hookrightarrow
    \lc \frakM_{ n}(X) \rc $$
induces a monoid embedding
\begin{equation*}
    \frakM_{ n}(X) = M \big( X\cup \lc\frakM_{n-1}(X) \rc \big)
 \hookrightarrow
    \frakM_{ n+1}(X) = M\big( X\cup\lc\frakM_{n}(X) \rc \big).
\end{equation*}
Finally we define the monoid
$$ \frakM (X):=\dirlim \frakM_n = \bigcup_{n\geq 0}\frakM_{ n}(X),$$
whose elements are called {\bf bracketed words on $X$}.  Elements $w\in \frakM_n\setminus \frakM_{n-1}$ are said to have {\bf depth} $n$,
denoted by $\dep(w) = n$.

\begin{lemma}({\bf \cite{Gub}})
Every bracketed word $w\neq1$ has a unique decomposition $w= w_{1}\cdots w_{m},$
where $w_{i},1\leq i\leq m$, is in $X$ or in $\lc\frakM(X)\rc:=\{\lc w\rc\mid w\in \frakM(X)\}$.
We call $\bre{w}:= m$ the {\bf breadth} of $w$.
\end{lemma}

The following result shows that $\frakM(X)$ is the free object in the category of operated monoids.

\begin{lemma}({\bf \cite{Gop, Gub}})
Let $i_X:X \to \frakM(X)$ be the natural embeddings. Then the triple $(\frakM(X),\lc\ \rc, i_X)$ is the free operated monoid on $X$,
where $$\lc \,\rc: \frakM(X) \to \frakM(X),\, w\mapsto \lc w\rc$$
is an operator on $\frakM(X)$.
\mlabel{lem:freetm}
\end{lemma}
The concept of $\star$-bracketed words plays a crucial role in
the theory of Gr\"{o}bner-Shirshov bases~\mcite{BC}.

\begin{defn}
Let $X$ be a set, $\star$ a symbol not in $X$ and $X^\star = X \cup \{\star\}$.
\begin{enumerate}
\item By a {\bf  $\star$-bracketed word} on $X$, we mean any bracketed word in $\mapm{X^\star}$ with exactly one occurrence of $\star$, counting multiplicities. The set of all $\star$-bracketed words on $X$ is denoted by $\frakM^{\star}(X)$.
\item For $q\in \frakM^\star(X)$ and $u \in  \frakM({X})$, we define $q|_{\star \mapsto u}$ to be the bracketed word on $X$ obtained by replacing the symbol $\star$ in $q$ by $u$, for convenience, denoted it by  $q|_{ u}$.
\item Let $u,v\in \frakM(X)$. We say that $u$ is a {\bf bracketed subword} of $v$, if there exist $q\in \frakM^{\star}(X) $ such that $v=q|_u$.

Generally, with $\star_1,\star_2$ distinct symbols not in
$X$, set $X^{\star 2} := X\cup \{\star_1, \star_2\}$.

\item We define an {$(\star_1,
    \star_2)${\bf -bracketed word} on $X$} to be a bracketed word in
  $\frakM(X^{\star 2})$ with exactly one occurrence of each of
  $\star_i$, $i=1,2$. The set of all
  $(\star_1, \star_2)$-bracketed words on $X$
  is denoted by $\frakM^{\star_1,\star_2}(X)$.

\item For $q\in \frakM^{\star_1,\star_2}(X)$ and $u_1,u_2 \in
   \frakM^{\star_1,\star_2}(X)$, we define
$$\stars{q}{u_1, u_2} := \stars{q}{\star_1 \mapsto u_1, \star_2 \mapsto u_2}$$
to be obtained by replacing the letters $\star_i$ in $q$ by $u_i$ for
$i=1,2$.
\end{enumerate}
\end{defn}

\begin{remark}
If $p|_u=p|_v$ with $p\in \frakM^{\star}(X)$ and $u,v\in\frakM(X)$,
then $u=v$ by the freeness of $\frakM(X)$.
\mlabel{re}
\end{remark}

The concept of operated congruences will be used throughout the paper.

\begin{defn}
An equivalence $R$ on $\frakM(X)$ is {\bf operated congruence} if

\begin{enumerate}
\item[(C1)]  $(\forall a,b,c\in \frakM(X))\  (a,b)\in R\Rightarrow(ac,bc)\in R;$

\item[(C2)]   $(\forall a,b,c\in \frakM(X)) \ (a,b)\in R\Rightarrow(ca,cb)\in R;$

\item[(C3)]   $(\forall a,b\in \frakM(X))\  (a,b)\in R\Rightarrow(\lc a \rc, \lc b \rc)\in R.$
\end{enumerate}
\mlabel{defn:ocong}
\end{defn}

Let $R$ be a binary relation on $\frakM(X)$. There is a unique smallest operated congruence $\ocong{R}$ on $\frakM(X)$ containing $R$, which will be described in the following. Define
\begin{equation}
R^c:=\{ (q|_{a}, q|_{b}) \mid q\in \frakM^{\star}(X), (a,b)\in R \}.
\mlabel{eq:rcc}
\end{equation}

We record some basic properties of $R^c$. For any $u\in \frakM(X)$, define recursively $\lc u\rc^{(0)}:= u$ and
$\lc u\rc^{(k+1)}:= \lc \lc u\rc^{(k)}\rc$ for $k\geq 0$.

\begin{lemma}
 $R^c$ is the smallest binary relation containing $R$ and satisfy (C1), (C2) and (C3).
 \mlabel{lem:Rc}
\end{lemma}

 \begin{proof}
According to the definition of $R^c$, we have $R\subseteq R^c$ by choosing $q=\star$.
Let $c\in \frakM(X)$ and
$$(q|_{a}, q|_{b})\in R^c \text{ with } q\in \frakM^\star(X), (a,b)\in R.$$
Write $q_1:= qc.$ Then by Eq.~(\mref{eq:rcc}),
$$((q|_{a})c, (q|_{b})c) = ( (qc)|_a, (qc)|_b) = (q_1|_a, q_1|_b) \in R^c.$$
So $R^c$ satisfies (C1). By symmetry, $R^c$ also satisfies ${\rm (C2)}$. To prove $R^c$ satisfies ${\rm (C3)}$, let $q_2:= \lc q\rc\in \frakM^\star(X)$. Then again from Eq.~(\mref{eq:rcc}),
$$(\lc q|_a\rc,  \lc q|_b\rc) = (q_2|_a, q_2|_b) \in R^c.$$

Suppose that $T$ is a binary relation containing $R$ and satisfying (C1), (C2) and (C3).
Let $(q|_a, q|_b)\in R^c$ with $q\in \frakM^{\star}(X)$ and $(a,b)\in R$.
Then $(a, b) \in T$ by $R\subseteq T$. Hence $(q|_a, q|_b) \in T$ by (C1), (C2) and (C3).
\end{proof}

For a binary relation $R$ on $\frakM(X)$, write $R^{-1}:=\{(b,a)\mid (a,b)\in R\} $.

 \begin{lemma}
Let $R,S$ be binary relations on $\frakM(X) $. Then
\begin{enumerate}
\item $R\subseteq S\Rightarrow R^c\subseteq S^c;$ \mlabel{it:brm1}

\item $(R^{-1})^c = (R^c)^{-1};$ \mlabel{it:brm2}

\item $(R\cup S)^c = R^c \cup S^c.$ \mlabel{it:brm3}
\end{enumerate}
\mlabel{lem:Rcp}
 \end{lemma}

\begin{proof}
(\mref{it:brm1}) This follows directly from Eq.~(\mref{eq:rcc}).

(\mref{it:brm2})
Let $(q|_a, q|_b)\in (R^{-1})^c$ with $(a,b)\in R^{-1}$ and $q\in\frakM^{\star}(X).$
Then
$$(b,a)\in R,\, (q|_b, q|_a)\in R^c\,\text{ and }\,  (q|_a, q|_b)\in (R^c)^{-1},$$
and so $(R^{-1})^c \subseteq  (R^c)^{-1}$.
Conversely, let $(q|_a,q|_b)\in (R^c)^{-1}$ with $q\in\frakM^{\star}(X).$ Then
$$(q|_b,q|_a)\in R^c,\, (b,a)\in R\,\text{ and }\,  (a,b)\in R^{-1},$$
which implies that
$$(q|_a, q|_b)\in (R^{-1})^c\,\text{ and }\,  (R^c)^{-1}\subseteq  (R^{-1})^c,$$
as needed.

(\mref{it:brm3})
By Item~(\mref{it:brm1}), we have
$$R^c\subseteq(R\cup S)^c\,\text{ and }\, S^c\subseteq(R\cup S)^c,$$
and thus $R^c\cup S^c\subseteq (R\cup S)^c.$ Conversely, let $(q|_a, q|_b)\in(R\cup S)^c$
with $(a,b)\in R\cup S$. Then
$$ (a,b)\in R \, \text{ or } \, (a,b)\in S  \, \text{ and so } \,  (q|_a, q|_b)\in R^c\,  \text{ or }\,(q|_a, q|_b)\in S^c.$$
Hence $(q|_a, q|_b)\in R^c\cup S^c$ and $(R\cup S)^c \subseteq R^c\cup S^c$.
\end{proof}

 \begin{lemma}
 Let $R$ be a binary relation on $\frakM(X) $ satisfying (C1), (C2) and (C3). Then, so is $R^n (=R\circ R\circ \cdots \circ R )$ for all $n\geq 1$.
 \mlabel{lem:Rn}
 \end{lemma}

 \begin{proof}
 Let $(a,b)\in R^n$. Then there exist $v_1,v_2,\cdots,v_{n-1}$ in $\frakM(X) $ such that $$(a,v_1),(v_1,v_2),\cdots,(v_{n-1},b)\in R.$$
 Because $R$ satisfies (C1),(C2) and (C3), it follows that, for all $c$ in $\frakM(X)$, $$(ca,cv_1),(cv_1,cv_2),\cdots,(cv_{n-1},cb)\in R,$$
 $$(ac,v_1c),(v_1c,v_2c),\cdots,(v_{n-1}c,bc)\in R,$$
 $$(\lc a \rc,\lc v_1 \rc),(\lc v_1 \rc,\lc v_2 \rc),\cdots,(\lc v_{n-1} \rc,\lc b \rc)\in R.$$
 So
 $$(ca,cb),(ac,ab),(\lc a \rc,\lc b \rc)\in R^n,$$ as required.
 \end{proof}

We recall the construction of equivalences $R^e$ generated by a binary relation $R$ on a set $X$~\cite{JMH}.
Write $$R^\infty :=\cup_{n\geq 1} R^n\,\text{ and }\, 1_X := \{(x,x) \mid x\in X\}.$$

\begin{lemma}(\cite{JMH})
Let $R$ be a binary relation on a set $X$. Then $R^e = [R\cup R^{-1}\cup 1_X]^\infty$. More precisely,
$(x,y)\in R^e$ if and only if either $x=y$ or,  for some $n$ in $\bf N$, there is a sequence of elements
$$x=z_1, z_2, \cdots,  z_n=y,$$  in which, for each $i$ in $\{1,2,\cdots,n-1\}$, either $(z_i,z_{i+1})\in R$ or $(z_{i+1},z_i)\in R.$
\mlabel{lem:equ}
\mlabel{lem:pro}
 \end{lemma}

Now we arrive at the position to give the description of operated congruences generated by binary relations.

 \begin{prop}
For every binary relation $R$ on $\frakM(X) $, $\ocong{R}  = (R^{c})^e.$
\mlabel{prop}
 \end{prop}

 \begin{proof}
By Lemma~\mref{lem:equ}, $(R^{c})^e$ is an equivalence containing $R^{c}$, and so certainly containing $R$ by Lemma~\mref{lem:Rc}.
Next we show that $(R^{c})^e$ satisfies (C1),(C2) and (C3).
From Lemma~\mref{lem:equ},
$$(R^{c})^e = S^\infty = \cup_{n\geq 1} S^n,\,\text{ where }\, S=R^c\cup (R^c)^{-1}\cup 1_{\frakM(X)}.$$
Using Lemma~\mref{lem:Rcp} and the fact that $ 1_{\frakM(X)} =  1_{\frakM(X)}^c$,
we get
$$S = R^c\cup (R^c)^{-1}\cup 1_{\frakM(X)} = R^c\cup (R^{-1})^{c}\cup 1_{\frakM(X)}^c=(R\cup R^{-1}\cup 1_{\frakM(X)})^c,$$
which implies from Lemma~\mref{lem:Rc} that
$S$ satisfies (C1), (C2) and (C3), and so is $S^n$ for all $n\geq 1$ by Lemma~\mref{lem:Rn}.
Hence
$$(R^c)^e = S^\infty = \cup_{n\geq 1} S^n $$ also satisfies (C1), (C2) and (C3).
Moreover since $(R^c)^e$ is already an equivalence, it is an operated congruence by Definition~\mref{defn:ocong}.

Finally, suppose $T$ is an arbitrary operated congruence on $\frakM(X)$ containing $R$.
Then $T^c=T$ by Definition~\mref{defn:ocong} and Eq.~(\mref{eq:rcc}). So from Lemma~\mref{lem:Rcp},
$$R^c\subseteq T^c=T\,\text{ and }\, (R^c)^e \subseteq T^e = T.$$
This completes the proof.
 \end{proof}



\begin{prop}
Let $R$ be a binary relation on $\frakM (X)$.
Then $(a,b)\in \ocong{R} $ if and only if either $a=b$ or,  for some $n$ in $\bf N$, there is a sequence of elements $a=v_1, v_2, \cdots,  v_n=b$,  in which, for each $i$ in $\{1,2,\cdots,n-1\}$, either $(v_i,v_{i+1})\in R^c$ or $(v_{i+1},v_i)\in R^c.$
\mlabel{pres}
\end{prop}

 \begin{proof}
It follows from Lemmas~\mref{lem:pro} and~\mref{prop}.
 \end{proof}

Every operated congruence $R$ on $\frakM(X)$ has a corresponding quotient structure $\frakM(X)/R$, whose elements are operated congruence classes for the relation. We end this subsection with an important concepts used later.

\begin{defn}
Let $X$ be a set and $R$ an operated congruence on $\frakM(X)$. We call $W\subseteq \frakM(X)$ a {\bf section} of $R$ if, for each operated congruence class $A$ of $\frakM(X)/R$, there exists exactly one element $w\in W$ such that $w\in A$.
\end{defn}

\subsection{Relationship between rewriting systems and sections}
In this subsection, we first associate a rewriting system $\Pi_S$ to a binary relation $S$ on $\frakM(X)$.
A monomial order compatible with all operations is needed.

\begin{defn}
Let $X$ be a set.
A {\bf  monomial order on $\frakM(X)$} is a
well-ordering $\leq$ on $\frakM(X)$ such that
\begin{equation}
u < v \Longrightarrow uw< vw,\, wu<wv,\,  \lc u\rc < \lc v\rc \  \text{for all } u, v, w \in \frakM(X).
%
\label{eq:morder}
\end{equation}
We denote $u < v$ if $u\leq v$ but $u\neq v$.
\label{de:morder}
\end{defn}

The following concepts are adapted from~\cite{BN, GGSZ}.



\begin{defn}
Let $X$ be a set and $\leq$ a monomial order on $\frakM(X)$.
Let $S$ be a binary relation on $\frakM (X)$.
\begin{enumerate}
\item {\bf A term-rewriting system} $\Pi_S$ on $\frakM(X)$ associated to $S$ is a binary relation of $\frakM (X)$, denote by
\begin{equation}
\Pi_S:=\{( q|_{t} , q|_{v}) \mid  q\in \frakM^{\star} (X), (t,v)\in S\cup S^{-1}, t>v \}.
\mlabel{eq:piss}
\end{equation}
An element in $ \Pi_S$ is called a {\bf rewriting rule }.

\item Let $f,g\in\frakM (X)$, we call {\bf $f$ rewrites to $g$ {\bf in one-step with respect to $\Pi_S$}}, if $(f,g)\in \Pi_S$.
We indicate any such one-step rewriting by $f\topi g$.

\item The reflexive transitive closure of $\Pi_S$ (as a binary relation on $\frakM(X)$) will be denoted by $\astarrowpi$ and we say $f$ {\bf rewrites to $g$ with respect to $\Pi_S$} if $f\astarrowpi g$.
In this case, we call $f$ is a {\bf predecessor} of $g$. Denote by $\pre(g)$ the set of all predecessors of $g$. Note that $g\in \pre(g)$.

\item Two elements  $f, g \in \frakM (X)$ are {\bf  joinable} if there exists $h \in \frakM (X)$ such that $f \astarrowpi h$ and $g \astarrowpi h$. Denote it by $f\downarrow_{\Pi_S}g$.

\item The image $\pi_1(\Pi_S)$ under the first projection map will be denoted by $\mathrm{Dom}(\Pi_S)$.
An element $f\in\frakM(X)$ is {\bf irreducible} or {\bf in normal form} if $ f \notin \mathrm{Dom}(\Pi_S)$, that is,
no more rewriting rule from $\Pi_S$ can apply to $f$.

\end{enumerate}
\mlabel{defn:rsts}
\end{defn}

\begin{defn}
The term-rewriting system $\Pi_S$ defined above is called
 \begin{enumerate}
 \item {\bf  terminating} if there is no infinite chain of one-setp rewriting
 \vspace{-3pt}$$f_0 \topi f_1 \topi f_2 \cdots;\vspace{-3pt}$$
\item {\bf  confluent}  if every fork $(f\astarrowpi h,  f\astarrowpi g)$, we have $g\downarrow_{\Pi_S} h$;
\item {\bf locally confluent} if for every local fork $(f\topi h,  f\topi g)$, we have $g\downarrow_{\Pi_S} h$;
\item {\bf  convergent} if it is both terminating and confluent.
\end{enumerate}
\label{def:ARS}
\end{defn}

A well-known result on rewriting systems is Newman's Lemma~\cite[Lemma 2.7.2]{BN}.

\begin{lemma}{\rm (Newman)}
 A terminating rewriting system is confluent if and only if it is locally confluent. \mlabel{lem:newman}
 \end{lemma}

\begin{lemma}
Let $X$ be a set and $\leq$ a monomial order on $\frakM(X)$.
Let $S$ be a binary relation on $\frakM (X)$.
The term-rewriting system $\Pi_S$ defined in Eq.~(\mref{eq:piss})  is terminating.
\mlabel{lem:termi}
\end{lemma}

\begin{proof}
Suppose to the contrary that $\Pi_S$ is not terminating. Then there is an infinite chain of one-step rewritings
 \vspace{-3pt}$$f_0 \topi f_1 \topi f_2 \cdots.\vspace{-3pt}$$
Let $f \topi g$ be a one-step rewriting. Then $(f,g)\in \Pi_S$ and so we can write $(f,g)=(q|_t, q|_v)$ with $q\in \frakM^{\star}(X)$, $(t,v)\in S\cup S^{-1}$
and $t>v$. Since $\leq$ is a monomial order, we have $f = q|_t > q|_v =g$. Hence
$$f_0 > f_1 > f_2 \cdots,$$
contradicting that $\leq$ is a well-order.
\end{proof}

We are going to capture the relationship between convergent term-rewriting systems on $\frakM(X)$ and
sections of quotients of $\frakM(X)$. Let us record three lemmas as a preparation.

\begin{lemma}
Let $X$ be a set and $\leq$ a monomial order on $\frakM(X)$.
Let $S$ be a binary relation on $\frakM (X)$.
If $a\astarrowpi b$ with $a,b\in \frakM(X)$, then $(a,b)\in \ocong{S}$.
\mlabel{lem:congr}
\end{lemma}

\begin{proof}
If $a =b$, then  $(a,b)\in \langle S\rangle$.
Suppose $a\neq b$. Then there are $a_1, a_2, \cdots, a_n\in \frakM(X)$ such that
$$a= a_1\topi a_2\topi \cdots \topi a_n = b,$$
which implies that
$$ (a_1,a_2), \cdots, (a_{n-1},a_n)\in \Pi_S \subseteq (S\cup S^{-1})^c=S^c\cup (S^c)^{-1}\subseteq (S^c)^e = \ocong{S}.$$
From the transitivity, we get $(a,b) = (a_1, a_n)\in \ocong{S}$, as required.
\end{proof}

%


\begin{lemma}
Let $X$ be a set and $\leq$ a monomial order on $\frakM(X)$.
Let $S$ be a binary relation on $\frakM (X)$.
For each operated congruence class $A$ of $\langle S \rangle$, we have
$$|A\cap \irr(S)|\geq 1\,\text{ and }\, A = \cup_{a\in A\cap \irr(S)} \pre(a),$$
where $\irr(S):=\frakM (X)\setminus  \mathrm{Dom}(\Pi_S)$ is the set of all irreducible elements under $\Pi_S$.
\mlabel{lem:classA}
\end{lemma}

\begin{proof}
Let $b\in A$. By Lemma~\mref{lem:termi}, there is $a\in \irr(S)$ such that $b\astarrowpi a$.
From Lemma~\mref{lem:congr}, we have $(b,a)\in \ocong{S}$ and so $a\in A$. Hence $a\in A\cap \irr(S)$
and $|A\cap \irr(S)|\geq 1$. Since  $b\astarrowpi a$, it follows that $b\in \pre(a)$ and thus
$$A\subseteq \cup_{a\in A\cap \irr(S)} \pre(a).$$
Conversely, for every element $a\in A$, we have $\pre(a)\subseteq A$ by Lemma~\mref{lem:congr} and so
$$\cup_{a\in A\cap \irr(S)} \pre(a)\subseteq A.$$
This completes the proof.
\end{proof}

\begin{lemma}
Let $X$ be a set and $\leq$ a monomial order on $\frakM(X)$.
Let $S$ be a binary relation on $\frakM(X)$.
For each operated congruence class $A$ of $\langle S \rangle$,  if $\Pi_S$ is confluent, then $|A\cap \irr(S)|=1$.
\mlabel{lem:confc}
 \end{lemma}

 \begin{proof}
Assume that $\Pi_S$ is confluent.
By Lemma~\mref{lem:classA}, we get
$ |A\cap \irr(S)|\geq 1.$
Suppose to the contrary that $|A\cap \irr(S)|\geq 2$. Let $A\cap \irr(S) = \{a_i\mid i\in I\}$ with $|I|\geq 2$. We have two cases to consider.

\noindent{\bf Case 1.} $\pre(a_i)\cap \pre(a_j)\neq \emptyset$  for some $i,j\in I$ with $i\neq j$. In this case, we can choose $b\in \pre(a_i)\cap \pre(a_j)$. Then $(b\astarrowpi a_i, b\astarrowpi a_j)$ is a fork, which is not joinable by $a_i, a_j\in \irr(S)$, contradicting that $\Pi_S$ is confluent.

\noindent{\bf Case 2.} $\pre(a_i) \cap \pre(a_j)= \emptyset$ for all $i,j\in I$ with $i\neq j$. We claim that
$(t_i, t_j), (t_j, t_i)\notin \Pi_S$ for all $t_i\in \pre(a_i)$ and $t_j\in \pre(a_j)$. Otherwise, by symmetry,
let $(t_i, t_j)\in \Pi_S$ for some $t_i\in \pre(a_i)$ and $t_j\in \pre(a_j)$. Then $t_i \topi t_j\astarrowpi a_j$
and so $t_i\in \pre(a_j) \cap \pre(a_i)$, contradicting that $\pre(a_i)\cap \pre(a_j) = \emptyset$.

Since $a_i$ and $a_j$ are in the same operated congruence class $A$, we have $(a_i,a_j)\in \langle S \rangle$.
By Lemma~\mref{pres}, there is a sequence $a_i=v_1, v_2, \cdots,  v_n=a_j$ with $n\geq 2$,  in which, for each $k$ in $\{1,2,\cdots,n-1\}$, either $(v_k,v_{k+1})\in S^c $ or $(v_{k+1},v_k)\in S^c.$ Note that $v_k\in A, 1\leq k \leq n$.
Because $a_i\in \irr(S)$, we can take $\ell:=\max\{ k \mid v_k\in \pre(a_i), 1\leq k\leq n\}$.
If $\ell = n$, then $a_j = v_n \in \calp(a_i)$ and so $a_j\in \calp(a_i)\cap \calp(a_j)$, a contradiction.
If $1\leq \ell\leq n-1$, then
$$v_\ell\in \pre(a_i), v_{\ell+1}\notin \pre(a_i)\,\text{ and }\, v_\ell \topi v_{\ell+1}.$$
Since $v_{\ell+1}\in A$, by Lemma~\mref{lem:classA}, there exists $p\in I$ with $p\neq i$ such that $v_{\ell+1} \astarrowpi a_p.$
Thus $v_\ell \astarrowpi a_i$ and $v_\ell \topi v_{\ell+1} \astarrowpi a_p$,
but $a_i, a_p\in \irr(S)$ and $a_i\neq a_p$, contradicting that $\Pi_S$ is confluent.
\end{proof}

Now we are ready for our main result of this section.

 \begin{theorem}
Let $S$ be a binary relation on $\frakM (X)$, and let $\leq$ be a monomial order on $\frakM(X)$ and $\Pi_S$ the term-rewriting system with respect to $\leq$.
Then $\irr(S)$ is a section of  $\langle S \rangle$   if and only if  $\Pi_S$   is convergent.
 \mlabel{thm:sect}
 \end{theorem}

 \begin{proof}
($\Leftarrow$)
Suppose $\Pi_S$ is convergent. Then in view of Lemma~\mref{lem:confc}, every operated congruence class intersects with $\irr(S)$ exactly one element. So $\irr(S)$ is a section of $\frakM (X){/}\langle S \rangle$ .

($\Rightarrow$)  Suppose to the contrary that $\Pi_S$ is not convergent.
Since $\Pi_S$ is terminating by Lemma~\mref{lem:termi}, $\Pi_S$ is not confluent.
 Then there exists a fork $(t\astarrowpi v_1, t\astarrowpi v_2)$ which is not joinable.
By Lemma~\mref{lem:termi},
$$v_1 \astarrowpi u_1\,\text{ and }\, v_2\astarrowpi u_2\,\text{ for some }\, u_1, u_2\in \irr(S), u_1\neq u_2.$$
which implies that $t\astarrowpi u_1$ and $t\astarrowpi u_2$.
By Lemma~\mref{lem:congr}, we have
$$ (t,  u_1),  (t, u_2) \in \langle S\rangle \,\text{ and so }  (u_1, u_2) \in \langle S\rangle.$$
Hence $u_1$ and $u_2$ in a same operated congruence class and they are in normal form, contradicting that $\irr(S)$ is a section of $\frakM (X){/}\langle S \rangle$.
\end{proof}

\section{Applications}
\mlabel{ss:sai}
Inverse monoids appear in a range of contexts, for example, they can be employed in the study of partial symmetries~\mcite{Law}.
$U$-monoids are natural generalizations of inverse monoids.

\begin{defn}
A $U$-monoid is a monoid $G$ equipped with a (unary) operator $^\circ$ such that $(u^\circ)^\circ = u$ for all $u\in G$.
\end{defn}

The following are two classes of $U$-monoids.

\begin{defn}
\begin{enumerate}
\item A $\ast$-monoid is a $U$-monoid $G$ satisfying the axiom $(uv)^\circ = v^\circ u^\circ$ for all $u,v\in G$. Such an operator is called an involution, and typically denoted by $\ast$.

\item A group is a $\ast$-monoid $G$ satisfying the axiom $u^\ast u =1=uu^\ast$ for all $u\in G$.
\end{enumerate}
\mlabel{defn:I*}
\end{defn}

In this section, as applications of Theorem~\mref{thm:sect}, we construct respectively
sections of free $\ast$-monoids and free groups, which are viewed as quotients of free operated monoids.
The monomial order given in~\cite[Lem.~5.3]{GSZ} will be used throughout this section.


\subsection{Free $\ast$-monoids} This subsection is spent to  construct sections of free $\ast$-monoids.
Recall that $\frakM(X)$ is the free operated monoid on $X$.
As a special case of a well-known result in universal algebra~\cite[Prop.1.3.6]{Coh},
the free $\ast$-monoid  on a set $X$ is the quotient of $\frakM(X)$ by the operated congruence $\langle S\rangle$, where
\begin{equation}
S =\{ \phi(w):= (\lc\lc w \rc\rc, w) ,\,  \psi(u, v):= ( \lc u v\rc,\lc v\rc  \lc u\rc) ,\,\omega :=(\lc 1\rc, 1) \mid w\in \frakM(X),u,v\in \frakM(X)\setminus \{1\} \}.
\mlabel{eq:defns}
\end{equation}

Before we go on to obtain a section of the free $\ast$-monoid, we recall a monomial order on $\frakM(X)$ from~\mcite{GSZ}.
Let $\leq$ be a well-ordering on a set $X$. It can be extended to a well-ordering on $\mapmonoid(X)=\dirlim \mapmonoid_n(X)$ by recursively defining a well-ordering $\leq_n$, on $\mapmonoid_n:=\mapmonoid_n(X)$ for each $n\geq 0$. When $n=0$, we have $\mapmonoid_0=M(X)$. In this case, we obtain a well-ordering by taking the shortlex order $\leq_{\text{sl}}$ on $M(X)$ induced by $\leq$ with the convention that $1 \leq_{\text{sl}}  u$ for all $u\in M(X)\backslash\{1\}$.
Suppose $\leq_n$ has been defined on $\mapmonoid_n:=M(X\sqcup \lc \mapmonoid_{n-1}\rc)$ for an $n\geq 0$. Denote by $\deg_{_{X}}(u)$ the number of $x\in X$ in $u$ with repetition. Then $\leq_n$ induces
\begin{enumerate}
\item
a well-ordering $\leq'_n$ on $\lc\mapmonoid_n\rc$ by
\begin{equation*}
\lc u\rc <'_n \lc v\rc \Longleftrightarrow u <_n v;
\end{equation*}
\item a well-ordering $\leq''_n$ on $X\sqcup \lc \mapmonoid_n\rc$;
\item a well-ordering $\leq'''_n$ on $X\sqcup \lc \mapmonoid_n\rc$ by
\begin{equation*}
u <'''_n v\Longleftrightarrow \left\{ \begin{array}{l} \text{\ either\ } \deg_{_X}(u) < \deg_{_{X}} (v) \\
\text{or\ } \deg_{_X}(u)=\deg_{_X}(v) \text{\ and\ } u <''_n v. \end{array} \right .
\end{equation*}
\item the shortlex well-ordering $\leq_{n+1}$ on $\mapmonoid_{n+1}=M(X\sqcup \lc \mapmonoid_n\rc)$ induced by $\leq_n'''$.
\end{enumerate}
The orders $\leq_n$ are compatible with the direct system $\{\mapmonoid_n\}_{n\geq 0}$ and hence induces a well-ordering, still denoted by $\leq$, on
$\mapmonoid(X)=\dirlim \mapmonoid_n$.

\begin{lemma}\cite[Lem.~5.3]{GSZ}
The order $\leq$ on $\mapmonoid(X)$ defined above is a monomial order.
\end{lemma}

Using this monomial order, we have
$$ w< \lc\lc w \rc\rc ,\, \lc v\rc \lc u \rc<\lc uv \rc, 1<\lc 1\rc\,\text{ for all }\, w\in \frakM(X), \text{ and}\, u,v\in \frakM(X)\setminus \{1\}.$$
For simplicity, if $\alpha$ is an element of a binary relation, we denote $\lbar{\alpha}$ and $\re{\alpha}$ by the domain and image of $\alpha$, respectively.
For example,
\begin{equation*}
\lbar{\phi(w)}= \lc\lc w \rc\rc,\, \re{\phi(w)}= w,\, \lbar{\psi(u,v)}=\lc u v \rc, \, \re{\varphi(u)}=\lc v\rc \lc u \rc,\, \text{ and }
\, \lbar{\omega}= \lc 1\rc,\, \re{\omega}= 1.
\end{equation*}
In the remainder of this paper, if $q|_{\lc\lc x\rc\rc}\topi q|_x$, we will indicate such rewriting step
in more detail by $q|_{\lc\lc x\rc\rc}\to_{\phi} q|_x$.
Similar notations will be used for $\to_{\varphi}$ and $\to_{\omega}$.

The following concept is finer than bracketed subwords, containing the information of placements~\mcite{ZheG}.

\begin{defn}
Let $X$ be a set, and let $w\in \mapm{X}$ be such that
\begin{equation}
\stars{q_1}{u_1}=w=\stars{q_2}{u_2} \ \text{ for some }u_1, u_2\in \mapm{X}, q_1, q_2\in \frakM^\star(X).
\label{eq:plas}
\end{equation}
\noindent
The two \plas $(u_1,q_1)$ and $(u_2,q_2)$ are called
\begin{enumerate}
\item
{\bf separated} if there exist $p \in \frakM^{\star_1,\star_2}(X)$ such that $\stars{q_1}{\star_1}=p|_{\star_1,\,u_2}$, $\stars{q_2}{\star_2}=p|_{u_1,\,\star_2}$, and $w=p|_{u_1,\,u_2}$;
\mlabel{it:bsep}

\item {\bf nested} if there exists $q \in \frakM^{\star}(X)$ such that either $q_2=\stars{q_1}q$ or $q_1=\stars{q_2}q$;
\mlabel{it:bnes}

\item  {\bf intersecting} if there exist $q \in \frakM^{\star}(X)$ and  $a, b, c \in \frakM(X)\backslash\{1\}$ such that $w=q|_{abc}$ and either
\begin{enumerate}
\item $ q_1=q|_{\star c}$ and $q_2=q|_{a\star}$; or

\item $q_1=q|_{a\star}$ and $q_2=q|_{\star c}$.
\end{enumerate}
\mlabel{it:bint}
\end{enumerate}
\mlabel{defn:bwrel}
\end{defn}

\begin{remark}
\begin{enumerate}
\item Suppose the placements $(u_1, q_1)$ and $(u_2,q_2)$ are nested. If $q_2={q_1}|_q$, then
$q_1|_{q|_{u_2}} = q_2|_{u_2} = w = q_1|_{u_1}$. By Remark~\mref{re}, we have $q|_{u_2}=u_1$, i.e., $u_2$ is a bracketed subword of $u_1$.
Similarly, if $q_1={q_2}|_q$, then  $q|_{u_1}=u_2$, i.e., $u_1$ is a bracketed subword of $u_2$.

\item Suppose the placements $(u_1, q_1)$ and $(u_2,q_2)$ are intersecting. If $ q_1=q|_{\star c}$ and $q_2=q|_{a\star}$, then
    $$q|_{u_1 c}=q_1|_{u_1}=w=q|_{abc}\,\, and \,\, q|_{a u_2}=q_2|_{u_2}=w=q|_{abc}.$$
Again by Remark~\mref{re}, $u_1 c=abc$ and $a u_2=abc$. So $u_1=ab$ and $u_2=bc.$
Similarly, if $q_1=q|_{a\star}$ and $q_2=q|_{\star c}$, then $u_1=ab$ and $u_2=bc$.
\end{enumerate}
\mlabel{rem}
\end{remark}

\begin{lemma}~{\rm \cite[Thm.~4.11]{ZheG}}
Let $w\in\frakM(X)$. For any two \plas $(u_1, q_1)$ and $(u_2,q_2)$ in $w$, exactly one of the following is true$\,:$
\begin{enumerate}
\item
$(u_1, q_1)$ and $(u_2,q_2)$ are separated$\,;$
\item
$(u_1, q_1)$ and $(u_2,q_2)$ are nested$\,;$
\item
$(u_1, q_1)$ and $(u_2,q_2)$ are intersecting. \mlabel{it:inter}
\end{enumerate}
\mlabel{lem:thrrel}
\end{lemma}

\begin{lemma}
Let $w\in\frakM(X)$. For any two placements $(u_1, q_1)$ and $(u_2, q_2)$ in $w$, if the breadth $|u_1|$ is 1, then $(u_1, q_1)$ and $(u_2, q_2)$  cann't  be intersecting.
\mlabel{lem:inter}
\end{lemma}

\begin{proof}
Suppose the placements $(u_1, q_1)$ and $(u_2, q_2)$ are intersecting.
From Definition~\mref{defn:bwrel}~(\mref{it:bint}), there are $q \in \frakM^{\star}(X)$ and  $a, b, c \in \frakM(X)\backslash\{1\}$ such that $w=q|_{abc}$ and either $q_1=q|_{\star c}$ and $q_2=q|_{a\star}$, or $q_1=q|_{a\star}$ and $q_2=q|_{\star c}$.
For the former case, we have $u_1=ab$ and $ u_2=bc$ by Remark~\mref{rem};
for the later case, we get $u_1=bc$ and $ u_2=ab.$
Since $a,b,c\in \frakM(X)\setminus\{1\}$, it follows that $|a|,|b|,|c|>1$ and so $|u_1|\geq 2$ in both cases, contradicting that $|u_1| =1$.
\end{proof}

\begin{lemma}
Let $S$ be the binary relation given in Eq.~(\mref{eq:defns}) and
$\alpha, \beta\in S$.
Suppose $\stars{q_1}{{\lbar{\alpha}}} = \stars{q_2}{\lbar{\beta}}$ for some $q_1, q_2\in \frakM^\star(X)$.
If the placements $({\lbar{\alpha}}, q_1)$ and $(\lbar{\beta}, q_2)$ are separated, then $\stars{q_1}{\re{\alpha}} \downarrow_{\Pi_S}\stars{q_2}{\re{\beta}}$.
\mlabel{lem:sepe}
\end{lemma}

\begin{proof}
In view of Definition~ \mref{defn:bwrel}~(\mref{it:bsep}), there exists $p\in \frakM^{\star_1,\star_2}(X)$ such that
$$
\stars{q_1}{\star_1}= \stars{p}{\star_1,\,\lbar{\beta}}\,\text{ and }\,
\stars{q_2}{\star_2}= \stars{p}{\lbar{\alpha},\,\star_2},
$$
whence
$$q_1|_{\re{\alpha}}=p|_{\re{\alpha},\,\lbar{\beta}} \topi p|_{\re{\alpha},\,\re{\beta}},$$
$$q_2|_{\re{\beta}}=p|_{\lbar{\alpha},\,\re{\beta}} \topi p|_{\re{\alpha},\,\re{\beta}}.$$
So we conclude that $\stars{q_1}{\re{\alpha}} \downarrow_{\Pi_S}\stars{q_2}{\re{\beta}}$.
\end{proof}

\begin{lemma}
Let $S$ be the binary relation given in Eq.~(\mref{eq:defns}) and
$\alpha, \beta\in S$.
Suppose $\stars{q_1}{{\lbar{\alpha}}} = \stars{q_2}{\lbar{\beta}}$  for some $q_1, q_2\in \frakM^\star(X)$.
If the placements $(\lbar{\alpha}, q_1)$ and $(\lbar{\beta}, q_2)$ are nested, then
$\stars{q_1}{\re{\alpha}} \downarrow_{\Pi_S} \stars{q_2}{\re{\beta}}.$
\mlabel{lem:nest}
\end{lemma}

\begin{proof}
Suppose the two placements $(\lbar{\alpha}, q_1)$ and $(\lbar{\beta}, q_2)$ are nested.
According to the choice of $\alpha$ and $\beta$, we have the following cases to consider.

\noindent{\bf Case 1.} $\lbar{\alpha} = \lbar{\phi(u)} =\lc\lc u \rc\rc$ and $\lbar{\beta}= \lbar{\phi(v)} = \lc\lc v \rc\rc.$
By symmetry, we may assume that $q_{1}=q_{2}|_{q}$ for some $q\in \frakM^\star(X)$.
Then by Remark~\mref{rem}, $q|_{\lc\lc u \rc\rc} = \lc\lc v \rc\rc$, i.e.,
$\lc\lc u \rc\rc$ is a bracketed subword of $\lc\lc v \rc\rc$.

\noindent{\bf Subcase 1.1.}  $\lc\lc u \rc\rc=\lc\lc v \rc\rc$. Then $\alpha = \beta$, $q= \star $ and $q_1=q_2$. Hence $\stars{q_1}{\re{\alpha}} = \stars{q_2}{\re{\beta}}$ and   $\stars{q_1}{\re{\alpha}} \downarrow_{\Pi_S} \stars{q_2}{\re{\beta}}$ trivially.

\noindent{\bf Subcase 1.2.}  $\lc\lc u \rc\rc = \lc v \rc$.
Then $v=\lc u \rc$, $q|_{\lc\lc u\rc\rc}=\lc\lc v\rc\rc=\lc\lc\lc u\rc\rc\rc=\lc\star\rc|_{\lc\lc u\rc\rc}$ and $q=\lc \star \rc$.
Hence
$$q_1|_{\re\alpha}=q_1|_u= (q_2|_q)|_u = q_2|_{q|_u} =  q_2|_{\lc \star \rc|_u} =  q_2|_{\lc u \rc} = q_2|_v  = q_2|_{\re\beta}$$
%
and so $\stars{q_1}{\re{\alpha}} \downarrow_{\Pi_S} \stars{q_2}{\re{\beta}}$ trivially.

\noindent{\bf Subcase 1.3.}  $\lc\lc u \rc\rc$ is a bracketed subword of $ v $.
Then there exists $p\in \frakM^\star(X)$ such that $ v = \stars{p}{\lc\lc u\rc\rc}$ and $ q = \lc\lc p\rc\rc$,
which implies
\begin{equation*}
q_1|_{\re\alpha}=q_1|_u= (q_2|_q)|_u = q_2|_{q|_u} =  q_2|_{\lc\lc p \rc\rc|_u} =  q_2|_{\lc\lc p|_u \rc\rc} \to_\phi q_2|_{ p|_u},
\end{equation*}
$$q_2|_{\re\beta}=q_2|_v= q_2|_{ p|_{\lc\lc u \rc\rc}} \to_\phi q_2|_{ p|_u}.$$
Consequently $\stars{q_1}{\re{\alpha}} \downarrow_{\Pi_S} \stars{q_2}{\re{\beta}}$.

\noindent{\bf Case 2.} $\lbar{\alpha}=\lbar{\psi(u',v')} = \lc u'v'\rc$ and $\lbar{\beta}= \lbar{\psi(u,v)} =\lc uv\rc.$
By symmetry, we may assume that $q_{1}=q_{2}|_{q}$ for some $q\in \frakM^\star(X)$.
Then by Remark~\mref{rem}, $q|_{\lc u'v'\rc} = \lc uv\rc$, i.e., $\lc u'v'\rc$ is a bracketed subword of $\lc uv\rc$.

\noindent{\bf Subcase 2.1.} $\lc u'v'\rc = \lc uv\rc$. This is similar to Subcase 1.1.

\noindent{\bf Subcase 2.2.} $\lc u'v'\rc$ is a bracketed subword of $u$ or $v$. By symmetry, we may assume $\lc u'v'\rc$ is a bracketed subword of $ u $. i.e., there exists $p\in \frakM^\star(X)$ such that $u = p|_{\lc u'v'\rc}$ and $q = \lc pv\rc$.
Whence
$$q_1|_{\re\alpha}=q_1|_{\lc v'\rc\lc u' \rc} = (q_2|_q)|_{\lc v'\rc\lc u' \rc} = q_2|_{q|_{\lc v'\rc\lc u' \rc}}  = q_2|_{\lc pv \rc|_{\lc v'\rc\lc u' \rc}}= q_2|_{\lc p|_{\lc v' \rc \lc u' \rc} v \rc}
\rightarrow_{\psi}q_2|_{ \lc v\rc\lc p|_{\lc v'\rc\lc u' \rc} \rc},$$
$$q_2|_{\re\beta}=q_2|_{\lc v\rc\lc u \rc} = q_2|_{\lc v\rc\lc p \rc|_{\lc u'v'\rc}} \rightarrow_\psi q_2|_{ {\lc v\rc\lc p \rc}|_{\lc v'\rc\lc u' \rc}}=q_2|_{ {\lc v\rc\lc p}|_{\lc v'\rc\lc u' \rc} \rc}.$$
Hence $\stars{q_1}{\re{\alpha}} \downarrow_{\Pi_S} \stars{q_2}{\re{\beta}}$.

\noindent{\bf Case 3.} $\lbar{\alpha}= \lbar{\phi(w)} =\lc\lc w \rc\rc$ and $\lbar{\beta}= \lbar{\psi(u,v)} =\lc uv\rc.$

\noindent{\bf Subcase 3.1.} $q_{1}=q_{2}|_{q}$ for some $q\in \frakM^\star(X)$.
By Remark~\mref{rem}, $q|_{\lc\lc w\rc\rc} = \lc uv\rc,$
that is,  $\lc\lc w\rc\rc$ is a bracketed subword of $\lc uv\rc$.

\noindent{\bf Subcase 3.1.1.} $\lc\lc w\rc\rc=\lc uv \rc$. Then $\lc w\rc=uv$ and so $u=1$ or $v=1$, which contradicts from
Eq.~(\mref{eq:defns2}) that ${\beta}\in S$.

\noindent{\bf Subcase 3.1.2.} $\lc\lc w\rc\rc$ is a bracketed subword of $uv$.
By symmetry, we can assume that  $\lc\lc w\rc\rc$ is a bracketed subword of $u$, i.e., there exists $p\in \frakM^\star(X)$ such that $u = p|_{\lc\lc w\rc\rc}$ and $q = \lc pv\rc$.
So
$$q_1|_{\re\alpha} = q_1|_{ w } = (q_2|_q)|_w = q_2|_{q|_w} = q_2|_{\lc pv\rc|_w} = q_2|_{\lc p|_w v\rc}  \to_\psi q_2|_{ {\lc v \rc \lc p|_{w} \rc}},$$
$$q_2|_{\re\beta}=q_2|_{\lc v \rc \lc u \rc} = q_2|_{\lc v \rc \lc p|_{\lc\lc w\rc\rc} \rc}  \to_\phi q_2|_{ {\lc v \rc \lc p|_{w} \rc}}.$$
Thus  $\stars{q_1}{\re{\alpha}} \downarrow_{\Pi_S} \stars{q_2}{\re{\beta}}$.

\noindent{\bf Subcase 3.2.}  $q_{2} = q_{1}|_{q}$ for some $q\in \frakM^\star(X)$.
By Remark~\mref{rem}, $\lc\lc w \rc\rc=q|_{\lc uv\rc},$ i.e.,
$\lc uv\rc$ is a bracketed subword of $\lc\lc w \rc\rc$.
If $\lc uv\rc=\lc\lc w \rc\rc$, similar to Case 3.1.1, we get $\stars{q_1}{\re{\alpha}} \downarrow_{\Pi_S} \stars{q_2}{\re{\beta}}$. Suppose $\lc uv\rc \neq \lc\lc w \rc\rc$.

\noindent{\bf Subcase 3.2.1.} $\lc uv\rc=\lc w \rc$. Then
$$uv=w,\, q|_{\lc uv\rc}=\lc \lc w \rc\rc=\lc\lc uv\rc\rc=\lc \star\rc|_{\lc uv \rc} \,\text{ and }\, q=\lc \star\rc.$$
Consequently $q_{1}|_{\re{\alpha}} =  q_{1}|_{ w }=q_1|_{uv}$ and
\begin{align*}
%
q_2|_{\re{\beta}} =& q_2|_{\lc v \rc \lc u \rc} = (q_1|_q)|_{\lc v \rc \lc u \rc} = q_1|_{q|_{\lc v \rc \lc u \rc}} = q_1|_{\lc \star\rc|_{\lc v \rc \lc u \rc}} =q_1|_{\lc  \lc v \rc \lc u \rc \rc}\\
\rightarrow_\psi& q_1|_{\lc\lc u \rc\rc \lc\lc v \rc\rc}\rightarrow_\phi  q_1|_{u\lc\lc v \rc\rc}\rightarrow_\phi  q_1|_{uv}.
\end{align*}
Hence $\stars{q_1}{\re{\alpha}} \downarrow_{\Pi_S} \stars{q_2}{\re{\beta}}$.

\noindent{\bf Subcase 3.2.2.} $\lc uv\rc$ is a bracketed subword of $ w$, i.e., there exists $p\in \frakM^\star(X)$ such that $w = p|_{\lc uv\rc}$ and $q = \lc\lc p\rc\rc$. Thus
\begin{align*}
q_{1}|_{\re\alpha}=&\, q_{1}|_{ w }= q_1|_{p|_{\lc uv\rc}} \to_\psi  q_1|_{p|_{\lc v \rc \lc u \rc}},\\
q_2|_{\re\beta}= q_2|_{\lc v \rc \lc u \rc} = (q_1|_q)|_{\lc v \rc \lc u \rc} =&\, q_1|_{q|_{\lc v \rc \lc u \rc}} = q_1|_{\lc  \lc p \rc\rc|_{\lc v \rc \lc u \rc}} = q_1|_{\lc  \lc p|_{\lc v \rc \lc u \rc} \rc\rc}  \rightarrow_\phi q_1|_{p|_{\lc v \rc \lc u \rc}},
\end{align*}
and so $\stars{q_1}{\re{\alpha}} \downarrow_{\Pi_S} \stars{q_1}{\re{\beta}}$.

\noindent{\bf Case 4.} $\lbar{\alpha}= \lbar{\phi(w)} =\lc\lc w \rc\rc$ and $\lbar{\beta}= \lbar{\omega} =\lc 1\rc.$

\noindent{\bf Subcase 4.1.} $q_{1}=q_{2}|_{q}$ for some $q\in \frakM^\star(X)$.
Then by Remark~\mref{rem}, $q|_{\lc\lc w\rc\rc} = \lc 1\rc,$
which implies that $\lc\lc w\rc\rc$ is a bracketed subword of $\lc 1\rc$, a contradiction by comparing the depth.

\noindent{\bf Subcase 4.2.}  $q_{2} = q_{1}|_{q}$ for some $q\in \frakM^\star(X)$.
By Remark~\mref{rem}, $\lc\lc w \rc\rc=q|_{\lc 1\rc},$ i.e.,
$\lc 1\rc$ is a bracketed subword of $\lc\lc w \rc\rc$. Note that $\lc 1\rc\neq\lc\lc w \rc\rc$.

\noindent{\bf Subcase 4.2.1.} $\lc 1\rc=\lc w \rc$. Then $w=1,$ $q|_{\lc 1\rc}=\lc\lc w\rc\rc=\lc\star\rc|_{\lc 1\rc}$ and  $q=\lc \star\rc$. Hence
\begin{align*}
q_{1}|_{\re\alpha}=&\, q_{1}|_{ w }=q_1|_{1},\\
q_2|_{\re\beta}= q_2|_{1} = (q_1|_q)|_{1}=& \, q_1|_{\lc \star\rc|_{1}} =q_1|_{\lc  1 \rc }\rightarrow_\omega q_1|_{1},
\end{align*}
and so  $\stars{q_1}{\re{\alpha}} \downarrow_{\Pi_S} \stars{q_2}{\re{\beta}}$.

\noindent{\bf Subcase 4.2.2.} $\lc 1\rc$ is a bracketed subword of $ w$, i.e., there exists $p\in \frakM^\star(X)$ such that $w = p|_{\lc 1\rc}$ and $q = \lc\lc p\rc\rc$. Then
\begin{align*}q_{1}|_{\re\alpha}=&\, q_{1}|_{ w }=q_1|_{p|_{\lc 1\rc}} \rightarrow_\omega  q_1|_{p|_{1}},\\
q_2|_{\re\beta}=&\, q_2|_{1} = (q_1|_q)|_{1}= q_1|_{\lc  \lc p \rc\rc|_{1}} \rightarrow_\phi q_1|_{p|_{1}}.
\end{align*}
Therefore $\stars{q_1}{\re{\alpha}} \downarrow_{\Pi_S} \stars{q_1}{\re{\beta}}$.

\noindent{\bf Case 5.} $\lbar{\alpha}=\lbar{\psi(u,v)}=\lc uv \rc$ and $\lbar{\beta}=\lbar{\omega}=\lc 1 \rc$. This is similar to Case 4.

\noindent{\bf Case 6.} $\lbar{\alpha}=\lbar{\omega}=\lc 1 \rc$ and $\lbar{\beta}=\lbar{\omega}=\lc 1 \rc$. This case is trivial since $\alpha$ and $\beta$ are equal.

This completes the proof.
\end{proof}

Now we arrive at our first main result of this section.

\begin{theorem}
Let $X$ be a set and $S$ the binary relation given in Eq.~(\mref{eq:defns}). With
the monomial order $\leq$ given in~\cite{GSZ}, we have
\begin{enumerate}
\item the \term-rewriting system $\Pi_S$  is convergent. \mlabel{it:conv1}

\item the set $\irr(S) = \frakM(X)\setminus {\rm Dom } (\Pi_S)$ is a section  of $\frakM(X)/\langle S \rangle$.  \mlabel{it:section1}
\end{enumerate}
\mlabel{thm:section}
\end{theorem}

\begin{proof}
(\mref{it:conv1})
Since $\leq$ is a monomial order on $\frakM(X)$, $\Pi_S$ is terminating by Lemma~\mref{lem:termi}.
So we are left to prove that $\Pi_S$ is confluent.
From Lemma~\mref{lem:newman}, it suffices to show that $\Pi_S$ is locally confluent. Let
$$ q_1|_{\re{\alpha}} \prescript{}{\Pi_S}\leftarrow q_1|_{\lbar{\alpha}}= w = q_2|_{\lbar{\beta}} \topi q_2|_{\re{\beta}} $$
be an arbitrary local fork, where
$q_1, q_2\in \frakM^\star(X), \alpha,\beta\in S.$
From Eq.~(\mref{eq:defns}), both of the breadth of $\lbar{\alpha}$ and $\lbar{\beta}$ are 1, and so
the placements $({\lbar{\alpha}}, q_1)$ and $(\lbar{\beta}, q_2)$  cann't be intersecting by Lemma~\mref{lem:inter}.
If the placements $({\lbar{\alpha}}, q_1)$ and $(\lbar{\beta}, q_2)$ are separated and nested, then $\stars{q_1}{\re{\alpha}} \downarrow_{\Pi_S}\stars{q_2}{\re{\beta}}$ by
Lemmas~\mref{lem:sepe} and~\mref{lem:nest}.

(\mref{it:section1}) It follows from Item~(\mref{it:conv1}) and Theorem~\mref{thm:sect}.
\end{proof}

\begin{remark}
It is well known~\cite{Law} that $\ast$-monoid is also called monoid with involution,
and the free monoid with involution $\ast$ on a set $X$ is the free monoid $M(X\cup X^\ast)$,
where $X^\ast:=  \{x^\ast \mid x\in X\}$ is a disjoint copy of $X$.
The $\irr(S)$ obtained in Theorem~\mref{thm:section} is precisely the set $M(X\cup X^\ast)$
if we identify $\lc x\rc$ with $x^\ast$ for each $x\in X$.
\end{remark}
%

\subsection{Free groups}
Again as a special case of a well-known result in universal algebra~\cite[Prop.1.3.6]{Coh},
the free group  on a set $X$ is the quotient of $\frakM(X)$ by the operated congruence $\langle S\rangle$, where
\begin{equation}
S :=\{(\lc\lc w \rc\rc, w) ,\,   (\lc u v\rc, \lc v\rc  \lc u \rc),\, (\lc w \rc w, 1),\, (w \lc w \rc, 1) \mid w\in\frakM(X), u,v\in \frakM(X)\setminus\{1\}\}.\mlabel{eq:defns2}
\end{equation}
In this subsection, we turn to construct a section of the free group on a set $W$.
Write
$$\phi(w):= (\lc\lc w \rc\rc, w) ,\,  \psi(u, v):= (\lc u v\rc, \lc v\rc  \lc u \rc),\,  \varphi(w):= (\lc w \rc w, 1),\,  \chi(w):= (w \lc w \rc, 1),$$
where $w\in\frakM(X),u,v\in \frakM(X)\setminus\{1\}$.
Note that if $w=1$, then $w\lc w\rc=1\lc 1\rc=\lc 1\rc$. So $(\lc 1\rc, 1)\in S$.
Here again under the monomial order $\leq$ given in~\mcite{GSZ}, we have
$$ w< \lc\lc w \rc\rc,\, \lc v \rc \lc u \rc < \lc uv \rc,\, 1< \lc w \rc w,\, 1< w\lc w \rc\,\text{ for }\, w\in\frakM(X),u,v\in \frakM(X)\setminus\{1\}.$$

\begin{lemma}
Let $S$ be the binary relation given in Eq.~(\mref{eq:defns2}) and
$\alpha, \beta\in S$.
Suppose $\stars{q_1}{{\lbar{\alpha}}} = \stars{q_2}{\lbar{\beta}}$ for some $q_1, q_2\in \frakM^\star(X)$.
If the placements $({\lbar{\alpha}}, q_1)$ and $(\lbar{\beta}, q_2)$ are separated, then $\stars{q_1}{\re{\alpha}} \downarrow_{\Pi_S}\stars{q_2}{\re{\beta}}$.
\mlabel{lem:sepe}
\end{lemma}

\begin{proof}
It is parallel to the proof of Lemma~\mref{lem:sepe}, because the proof of Lemma~\mref{lem:sepe} does not depend on the concrete expressions of $\alpha$ and $\beta$.
\end{proof}

\begin{lemma}
Let $S$ be the binary relation given in Eq.~(\mref{eq:defns2}) and
$\alpha, \beta\in S$.
Suppose $\stars{q_1}{{\lbar{\alpha}}} = \stars{q_2}{\lbar{\beta}}$ for some $q_1, q_2\in \frakM^\star(X)$.
If the placements $(\lbar{\alpha}, q_1)$ and $(\lbar{\beta}, q_2)$ are intersecting, then $\stars{q_1}{\re{\alpha}} \downarrow_{\Pi_S} \stars{q_2}{\re{\beta}}$.

\mlabel{lem:int}
\end{lemma}

\begin{proof}
Note that $\lbar {\phi(w)}=\lc \lc w \rc\rc$, $\lbar {\psi(u,v)}=\lc uv \rc$
and $|\lbar {\phi(u)}| =|\lbar {\psi(u,v)}|= 1$ for all $w\in \frakM(X), u, v\in \frakM(X)\setminus\{1\}$.
It follows from Lemma~\mref{lem:inter} that $\alpha,\beta\in \{ \varphi(u), \chi(u) \mid u\in \frakM(X) \}$.
Let $w:=q_1|_{\lbar{\alpha}} = q_2|_{\lbar{\beta}} $.
Then by Definition~\mref{defn:bwrel}(\mref{it:bint}), there are $q\in\frakM^{\star}(X)$ and $a, b, c\in \frakM(X)\setminus\{1\}$ such that $w=q|_{abc}$.
Depending on the forms of $\alpha$ and $\beta$, there are three cases to consider.

\noindent{\bf Case 1.} $\lbar{\alpha}= \lbar{\varphi(u)} = \lc u \rc u$ and $\lbar{\beta}= \lbar{\varphi(v)} =\lc v \rc v$ for some $u,v\in \frakM(X)$.
By symmetry, we may assume $q_1=q|_{\star c}$ and $q_2=q|_{a \star}.$
Using Remark~\mref{rem},  we get $\lc u\rc u=ab$ and $\lc v\rc v =bc.$
Since $|b|\geq 1$, $\lc v\rc$ is a bracketed subword of $b$. Suppose $v = v_1v_2$ and $b = \lc v\rc v_1$. Then $c = v_2$. Similarly from $\lc u\rc u = ab$, we can assume
$a = \lc u\rc u_1$ and $b=u_2$ with $u=u_1 u_2$. Then
$$u_2 = b = \lc v\rc v_1 = \lc v_1v_2\rc v_1\,\text{ and }\, a = \lc u\rc u_1 = \lc u_1 u_2\rc u_1 =   \lc u_1 \lc v_1v_2\rc v_1\rc u_1$$
and so
\begin{align*}
q_1|_{\re{\alpha}}=& q_1|_1=q|_{{\star c}|_1}=q|_c=q|_{v_2},\\
q_2|_{\re{\beta}}=& q_2|_1=q|_{{a \star }|_1}=q|_a=q|_{\lc u_1\lc v_1 v_2\rc v_1\rc u_1}\rightarrow_{\psi} q|_{\lc \lc v_1 v_2\rc v_1 \rc \lc u_1\rc u_1}\\
\rightarrow_{\varphi}& q|_{\lc \lc v_1 v_2\rc v_1 \rc}\rightarrow_{\psi} q|_{\lc \lc v_2\rc\lc v_1\rc v_1 \rc}\rightarrow_{\varphi} q|_{\lc \lc  v_2\rc \rc}\rightarrow_{\phi} q|_{v_2},
\end{align*}
which implies  $\stars{q_1}{\re{\alpha}} \downarrow_{\Pi_S} \stars{q_2}{\re{\beta}}.$

\noindent{\bf Case 2.} $\lbar{\alpha}= \lbar{\chi(u)} = u\lc u \rc$ and $\lbar{\beta}=\lbar{\chi(v)} = v\lc v \rc$ for some $u,v\in \frakM(X)$. This is similar to Case 1.

\noindent{\bf Case 3.} $\lbar{\alpha} = \lbar{\varphi(u)}  =\lc u \rc u$ and $\lbar{\beta}= \lbar{\chi(v)} = v\lc v \rc$,
or $\lbar{\alpha} = \lbar{\chi(u)}  = u\lc u \rc$ and $\lbar{\beta}= \lbar{\varphi(v)} = \lc v \rc v$ for some $u,v\in \frakM(X)$.
By symmetry, it suffices to consider the former case. Then according to Definition~\mref{defn:bwrel}(\mref{it:bint}), we have two subcases
to consider.

\noindent{\bf Subcase 3.1.} $q_1=q|_{\star c}$ and $q_2=q|_{a\star}.$
From Remark~\mref{rem}, $\lc u\rc u=ab$ and $ v \lc v\rc =bc.$
With a similar argument to Case 1, we can assume
\begin{align*}
b&=v_1,\, c=v_2\lc v_1 v_2\rc\, \text{ with }\, v=v_1v_2, \\
a&= \lc u_1 u_2\rc u_1, \, b= u_2 \, \text{ with }\, u=u_1u_2.
\end{align*}
Thus
\begin{align*}
q_1|_{\re{\alpha}}= & q_1|_1=q|_{{\star c}|_1}=q|_c=q|_{v_2\lc v_1 v_2\rc}\rightarrow_{\psi}q|_{v_2\lc v_2\rc \lc v_1 \rc}\rightarrow_{\chi}q|_{\lc v_1 \rc},\\
q_2|_{\re{\beta}}=& q_2|_1=q|_{{a \star }|_1}=q|_a=q|_{\lc u_1 u_2\rc u_1}\rightarrow_{\psi}q|_{\lc u_2\rc \lc u_1\rc u_1}\rightarrow_{\varphi}q|_{\lc u_2 \rc}=q|_{\lc v_1 \rc},
\end{align*}
and so $\stars{q_1}{\re{\alpha}} \downarrow_{\Pi_S} \stars{q_2}{\re{\beta}}.$

\noindent{\bf Subcase 3.2.} $q_1=q|_{a\star}$ and $ q_2=q|_{\star c}.$
From Remark~\mref{rem}, $\lc u\rc u=bc$ and $ v \lc v\rc =ab.$
Again similar to Case 1, we may suppose
$$ b=\lc u_1 u_2\rc u_1\,\text{ and }\,  c=u_2 \, \text{ with }\, u=u_1u_2. $$
Then $ a\lc u_1 u_2\rc u_1 = ab=v\lc v\rc.$
If $u_1\neq 1,$ then $\lc v \rc$ is a bracketed subword of $u_1$ and $a\lc u_1 u_2 \rc$ is a bracketed subword of $v$. So we get
 $$\dep(\lc v \rc) \leq \dep( u_1) < \dep(a\lc u_1 u_2 \rc ) \leq\dep(v),$$
a contradiction.
So
$$u_1=1, u_2 = u,  b=\lc u_1u_2\rc=\lc u\rc, c=u, a\lc u\rc = ab =v \lc v \rc,$$
which implies that $a=v$, $\lc u\rc = \lc v\rc$ and $u=v$. Thus
\begin{align*}
q_1|_{\re{\alpha}}=q_1|_1=q|_{{a \star}|_1}=q|_a=q|_v,\\
q_2|_{\re{\beta}}=q_2|_1=q|_{{ \star c}|_1}=q|_c=q|_u=q|_v,
\end{align*}
and so $\stars{q_1}{\re{\alpha}} \downarrow_{\Pi_S} \stars{q_2}{\re{\beta}}.$
\end{proof}

\begin{lemma}
Let $S$ be the binary relation given in Eq.~(\mref{eq:defns2}) and
$\alpha, \beta\in S$.
Suppose $\stars{q_1}{{\lbar{\alpha}}} = \stars{q_2}{\lbar{\beta}}$ for some $q_1, q_2\in \frakM^\star(X)$.
If the placements $(\lbar{\alpha}, q_1)$ and $(\lbar{\beta}, q_2)$ are nested, then
$\stars{q_1}{\re{\alpha}} \downarrow_{\Pi_S} \stars{q_2}{\re{\beta}}$.
\mlabel{lem:ne}
\end{lemma}

\begin{proof}
By Definition~\mref{defn:bwrel}(\mref{it:bint}), we can assume $q_1|_{\lbar{\alpha}}=w=q_2|_{\lbar{\beta}}$ for some $w\in \frakM(X).$
From Eq.~(\mref{eq:defns2}), there are four choices for each $\alpha$ and $\beta$.
In view of symmetry, there are ten pairs of $\alpha$ and $\beta$ to consider.
If $\alpha, \beta \in \{\phi(w), \psi(u, v)\mid w\in\frakM(X), u,v \in \frakM(X)\setminus\{1\}\},$ the result follows from Lemma~\mref{lem:nest}.
So three cases have been done and we are left to consider the following seven cases.

\noindent{\bf Case 1.} $\lbar{\alpha} =\phi(u)=\lc\lc u \rc\rc$ and $\lbar{\beta} =\varphi(v)=\lc v \rc v.$
Since either $\lbar{\alpha}$ is subword of $\lbar{\beta}$ or $\lbar{\beta}$ is subword of $\lbar{\alpha}$, we have the following two subcases.

\noindent{\bf Subcase 1.1.} $q_2=q_1|_q$ for some $q\in \frakM^\star(X).$
By Remark~\mref{rem}, ${\lc\lc u \rc\rc}=q|_{\lc v \rc v},$ i.e., $\lc v \rc v$ is a bracketed subword of $\lc\lc u \rc\rc$.
Note that $\lc v\rc v\neq \lc\lc u\rc\rc, \lc u\rc$ by comparing the breadth.
So $\lc v \rc v$ is a bracketed subword of $u$, i.e., there exists $p\in\frakM^\star(X)$ such that $u=p|_{\lc v \rc v }$ and $q=\lc \lc p \rc\rc$. Thus
\begin{align*}
q_1|_{\re{\alpha}}=& q_1|_u=q_1|_{p|_{\lc v \rc v}}\rightarrow_{\varphi}q_1|_{p|_1},\\
q_2|_{\re{\beta}}=& q_2|_1=q_1|_{q|_{1}}=q_1|_{\lc\lc p\rc\rc|_1}=q_1|_{\lc \lc p|_1\rc\rc}\rightarrow_{\phi}q_1|_{p|_1},
\end{align*}
and so $\stars{q_1}{\re{\alpha}} \downarrow_{\Pi_S} \stars{q_2}{\re{\beta}}$.

\noindent{\bf Subcase 1.2.} $q_1=q_2|_q$ for some $q\in \frakM^\star(X).$ Then
by Remark~\mref{rem}, $q|_{\lc\lc u \rc\rc}=\lc v\rc v,$ i.e., $\lc\lc u \rc\rc$ is a bracketed subword of $\lc v\rc v$. Note that
$\lc\lc u\rc\rc\neq \lc v\rc v$. So there are two points to consider.

\noindent{\bf Subcase 1.2.1.} $\lc\lc u \rc\rc=\lc v\rc$. Then $v=\lc u \rc.$ Since
$q|_{\lc\lc u \rc\rc}=\lc v\rc v= \lc\lc u \rc\rc\lc u \rc$, we have $q=\star \lc u \rc$. Thus
$$q_1|_{\re{\alpha}}=q_1|_u=q_2|_{q|_{u}}=q_2|_{u\lc u \rc}\rightarrow_{\chi}q_2|_{1}\, \,\text{and} \,\,q_2|_{\re{\beta}}=q_2|_{1}$$
and so $\stars{q_1}{\re{\alpha}} \downarrow_{\Pi_S} \stars{q_2}{\re{\beta}}$.

\noindent{\bf Subcase 1.2.2.} $\lc\lc u \rc\rc$ is a bracketed subword of  $v$, i.e., there exists $p\in \frakM^\star(X)$ such that $v=p|_{\lc\lc u \rc\rc}$. Then $q=\lc p \rc v$ or $q=\lc v\rc p$.
If $q=\lc p \rc v$, then
$$q_1|_{\re{\alpha}}=q_1|_u=q_2|_{q|_{u}}=q_2|_{\lc p|_u \rc v}=q_2|_{\lc p|_u \rc p|_{\lc \lc u \rc\rc} }\rightarrow_{\phi}q_2|_{\lc p|_u \rc p|_{u}}\rightarrow_{\varphi}q_2|_1.$$
If $q=\lc v \rc p$, then
$$q_1|_{\re{\alpha}}=q_1|_u=q_2|_{q|_{u}}=q_2|_{\lc v \rc p|_u}=q_2|_{\lc p|_{\lc \lc u \rc\rc} \rc p|_{ u} }\rightarrow_{\phi}q_2|_{\lc p|_u \rc p|_{u}}\rightarrow_{\varphi}q_2|_1.$$
Note that $q_2|_{\re{\beta}}=q_2|_{1}.$ So we conclude $\stars{q_1}{\re{\alpha}} \downarrow_{\Pi_S} \stars{q_2}{\re{\beta}}$.

\noindent{\bf Case 2.} $\lbar{\alpha} =\phi(u)=\lc\lc u \rc\rc$ and $\lbar{\beta} = \chi(v)=v\lc v \rc.$ This is similar to Case 1.

\noindent{\bf Case 3.} $\lbar{\alpha} =\psi(u_1,u_2)=\lc u_1 u_2 \rc$ and $\lbar{\beta} =\varphi(v)=\lc v \rc v.$
Again since either $\lbar{\alpha}$ is subword of $\lbar{\beta}$ or $\lbar{\beta}$ is subword of $\lbar{\alpha}$, there are two subcases to consider.

\noindent{\bf Subcase 3.1.} $q_1=q_2|_q$ for some $q\in \frakM(X)$. Then by Remark~\mref{rem}, $q|_{\lc u_1 u_2 \rc}=\lc v \rc v$, i.e., $\lc u_1 u_2 \rc$ is  a bracketed subword of $\lc v \rc v$. Note that $\lc v \rc v\neq\lc u_1 u_2\rc$ by comparing the breadth.

\noindent{\bf Subcase 3.1.1.} $\lc u_1 u_2\rc=\lc v \rc$. Then
$$v=u_1 u_2,\, q|_{\lc u_1 u_2 \rc}=\lc v \rc v=\lc u_1 u_2 \rc  u_1 u_2\,\text{ and }\, q=\star u_1 u_2$$
and so
$$q_1|_{\re{\alpha}}=q_1|_{\lc u_2\rc\lc u_1 \rc}=q_2|_{q|_{\lc u_2\rc\lc u_1 \rc}}=q_2|_{\star u_1u_2|_{\lc u_2\rc\lc u_1 \rc}}=q_2|_{\lc u_2\rc\lc u_1\rc u_1 u_2}\rightarrow_{\varphi}q_2|_{\lc u_2 \rc u_2}\rightarrow_{\varphi}q_2|_1.$$
Since $q_2|_{\re{\beta}}=q_2|_{1},$  we get $\stars{q_1}{\re{\alpha}} \downarrow_{\Pi_S} \stars{q_2}{\re{\beta}}$.

\noindent{\bf Subcase 3.1.2.} $\lc u_1 u_2\rc$ is a bracketed subword of  $v$, i.e., there exists $p\in \frakM^\star(X)$ such that $v=p|_{\lc u_1 u_2\rc}$. Then $q=\lc p \rc v$ or $q=\lc v\rc p$.
If $q=\lc p \rc v$, then
$$q_1|_{\re{\alpha}}=q_1|_{\lc u_2\rc\lc u_1\rc}=q_2|_{q|_{\lc u_2\rc\lc u_1\rc}}=q_2|_{\lc p|_{\lc u_2\rc\lc u_1\rc} \rc v}=q_2|_{\lc p|_{\lc u_2\rc\lc u_1\rc} \rc p|_{\lc u_1 u_2\rc} }\rightarrow_{\psi}q_2|_{\lc p|_{\lc u_2\rc\lc u_1\rc} \rc p|_{{\lc u_2\rc\lc u_1\rc}}}\rightarrow_{\varphi}q_2|_1,$$
If $q=\lc v \rc p$, then
$$q_1|_{\re{\alpha}}=q_1|_{\lc u_2\rc\lc u_1\rc}=q_2|_{q|_{\lc u_2\rc\lc u_1\rc}}=q_2|_{\lc v \rc p|_{\lc u_2\rc\lc u_1\rc}}=q_2|_{\lc p|_{\lc u_1 u_2\rc}\rc  p|_{\lc u_2\rc\lc u_1\rc} }\rightarrow_{\psi}q_2|_{\lc p|_{\lc u_2\rc\lc u_1\rc} \rc p|_{\lc u_2\rc\lc u_1\rc}}\rightarrow_{\varphi}q_2|_1.$$
Since $q_2|_{\re{\beta}}=q_2|_{1},$ we conclude that $\stars{q_1}{\re{\alpha}} \downarrow_{\Pi_S} \stars{q_2}{\re{\beta}}$.

\noindent{\bf Subcase 3.2.} $q_2=q_1|_q$ for some $q\in \frakM^{\star}(X)$. Then by Remark~\mref{rem}, $q|_{\lc v \rc v}=\lc u_1 u_2 \rc$, i.e., $\lc v \rc v$ is a bracketed subword of $\lc u_1 u_2 \rc$. Note that $\lc v \rc v\neq\lc u_1 u_2\rc.$
Thus we can assume $\lc v\rc v$ is a bracketed subword of $u_1u_2$.

\noindent{\bf Subcase 3.2.1.} $\lc v\rc v$ is a bracketed subword of $u_1$ or $u_2$. By symmetry, we only need to
consider the former. Then there exists $p\in \frakM^{\star}(X)$ such that $u_1=p|_{\lc v\rc v}$ and $q=\lc p u_2\rc.$ Consequently,
\begin{align*}
q_1|_{\re{\alpha}}=q_1|_{\lc u_2\rc\lc u_1\rc}=q_1|_{\lc u_2\rc\lc p|_{\lc v \rc v}\rc}\rightarrow_{\varphi} q_1|_{\lc u_2\rc\lc p|_1\rc},\\
q_2|_{\re{\beta}}=q_2|_{1}=q_1|_{q|_{1}}=q_1|_{\lc p|_{1}u_2\rc}\rightarrow_{\psi}q_1|_{\lc u_2\rc\lc p|_1\rc},
\end{align*}
and so $\stars{q_1}{\re{\alpha}} \downarrow_{\Pi_S} \stars{q_2}{\re{\beta}}$.

\noindent{\bf Subcase 3.2.2.} $\lc v\rc v$ is neither a bracketed subword of $u_1$ nor a bracketed subword of $u_2$.
Then
$$v=v_1 v_2, \, u_1 = u_1' \lc v\rc v_1\,\text{ and }\, u_2 = v_2 u_2'$$ for some $v_1, v_2, u_1', u_2'\in \frakM(X)$ with $v_2\neq 1$.
Consequently,
$$\lc u_1'\star u_2' \rc |_{\lc v\rc v} = \lc u_1'\lc v\rc v u_2' \rc =  \lc u_1'\lc v\rc v_1 v_2 u_2' \rc  = \lc u_1 u_2\rc = q|_{\lc v\rc v}
\, \text{ and }\, \lc u_1'\star u_2' \rc  = q.$$
Thus
\begin{align*}
q_1|_{\re{\alpha}}&= q_1|_{\lc u_2\rc\lc u_1\rc}= q_1|_{\lc v_2 u_2'\rc\lc u_1' \lc v_1 v_2\rc v_1 \rc}  \rightarrow_{\psi}
q_1|_{\lc v_2 u_2'\rc\lc u_1' \lc v_2 \rc \lc v_1\rc v_1 \rc} \rightarrow_{\varphi} q_1|_{\lc v_2 u_2'\rc\lc u_1' \lc v_2 \rc \rc} \\
&\rightarrow_{\psi} q_1|_{\lc u_2'\rc \lc v_2\rc \lc u_1' \lc v_2 \rc \rc} \rightarrow_{\psi} q_1|_{\lc u_2'\rc \lc v_2\rc \lc \lc v_2 \rc \rc \lc u_1' \rc}
\rightarrow_{\phi} q_1|_{\lc u_2'\rc \lc v_2\rc v_2\lc u_1' \rc} \rightarrow_{\varphi} q_1|_{\lc u_2'\rc \lc u_1' \rc},
 \\
q_2|_{\re{\beta}}&=q_2|_{1}=q_1|_{q|_{1}}=q_1|_{\lc u_1'u_2'\rc}\rightarrow_{\psi}q_1|_{\lc u_2'\rc \lc u_1' \rc},
\end{align*}
and so $\stars{q_1}{\re{\alpha}} \downarrow_{\Pi_S} \stars{q_2}{\re{\beta}}$.

\noindent{\bf Case 4.} $\lbar{\alpha} =\psi(u_1,u_2)=\lc u_1 u_2 \rc$ and $\lbar{\beta} = \chi(v)=v\lc v \rc.$ This is similar to Case 3.

\noindent{\bf Case 5.} $\lbar{\alpha} =\varphi(u)=\lc u \rc u$ and $\lbar{\beta} =\varphi(v)=\lc v \rc v.$
By symmetry, we may assume $q_1=q_2|_q$ for some $q\in \frakM^{\star}(X)$. From Remark~\mref{rem}, $q|_{\lc u \rc u}=\lc v \rc v$, i.e., $\lc u \rc u$ is a bracketed subword of $\lc v \rc v$.
If $\lc u \rc u=\lc v \rc v$, then $u = v$, $\alpha = \beta$, $q_1 = q_2$ and so $\stars{q_1}{\re{\alpha}} \downarrow_{\Pi_S} \stars{q_2}{\re{\beta}}$.
Suppose $\lc u \rc u \neq \lc v \rc v$. Since $\lc u \rc u\neq\lc v\rc$, $\lc u\rc u$ is a bracketed subword of $v$, i.e., there exists $p\in \frakM^{\star}(X)$ such that $v=p|_{\lc u\rc u}$. Then $q=\lc p \rc v$ or $q=\lc v \rc p$.
If $q=\lc p \rc v,$ then
$$q_1|_{\re{\alpha}}=q_1|_{1}=q_2|_{q|_{1}}=q_2|_{\lc p\rc|_{1}v}=q_2|_{\lc p|_{1}\rc p|_{\lc u\rc u}}\rightarrow_{\varphi} q_2|_{\lc p|_1\rc  p|_1}\rightarrow_{\varphi}q_2|_1.$$
If $q=\lc p \rc v,$ then
$$q_1|_{\re{\alpha}}=q_1|_{1}=q_2|_{q|_{1}}=q_2|_{\lc v\rc p|_{1}}=q_2|_{\lc p|_{\lc u\rc u}\rc p|_{1 }}\rightarrow_{\varphi} q_2|_{\lc p|_1\rc  p|_1}\rightarrow_{\varphi}q_2|_1.$$
Since $q_2|_{\re{\beta}}=q_2|_{1}$, we conclude that $\stars{q_1}{\re{\alpha}} \downarrow_{\Pi_S} \stars{q_2}{\re{\beta}}$.

\noindent{\bf Case 6.} $\lbar{\alpha} =\chi(u)=u\lc u\rc $ and $\lbar{\beta} =\chi(v)= v\lc v \rc.$ This is similar to Case 5.

\noindent{\bf Case 7.} $\lbar{\alpha}=\varphi(u)=\lc u\rc u$ and $\lbar{\beta} =\chi(v)=v \lc v \rc .$ This is also similar to Case 5.
\end{proof}

\begin{theorem}
Let $X$ be a set and $S$ the binary relation given in Eq.~(\mref{eq:defns2}). With the monomial order given in~\cite{GSZ}, we have
\begin{enumerate}
\item the term-rewriting system $\Pi_S$  is convergent. \mlabel{it:conv2}

\item the set $\irr(S) = \frakM(X)\setminus {\rm Dom } (\Pi_S)$ is a section of the free group $\frakM(X)/\langle S \rangle$.  \mlabel{it:section2}
\end{enumerate}
\mlabel{thm:sectionI}
\end{theorem}

\begin{proof}
(\mref{it:conv2}) With a similar argument to the proof of Theorem~\mref{thm:section}(\mref{it:conv1}),
the result follows from Lemmas~\mref{lem:sepe}, \mref{lem:int} and~\mref{lem:ne}.

(\mref{it:section2}) This part follows from Item~(\mref{it:conv1}) and Theorem~\mref{thm:section}.
\end{proof}

\begin{remark}
It is well known that reduced words are elements in the free group on a set $X$~\cite{Hung}.
The set $\irr(S) = \frakM(X)\setminus {\rm Dom } (\Pi_S)$ obtained in Theorem~\mref{thm:sectionI}, of course,  coincides with the set of reduced words.
\end{remark}

\smallskip

\noindent {\bf Acknowledgements}:
 The authors are supported by the National Natural Science Foundation of
China (No.~11771191), the Fundamental Research Funds for the Central
Universities (No.~lzujbky-2017-162), and the Natural Science Foundation of Gansu Province (Grant
No.~17JR5RA175) and Shandong Province (No. ZR2016AM02).

We thank the anonymous referee for valuable suggestions helping to improve the paper.

\end{document}